\documentclass[11pt,a4paper]{article}
\usepackage{amsmath,amssymb,amsthm}
\usepackage{soul}
\usepackage{color}
\usepackage{graphicx} 
\usepackage{subfigure} 

\usepackage{enumerate}

\newtheorem{theorem}{Theorem}[section]
\newtheorem{lemma}[theorem]{Lemma}

\newtheorem{proposition}[theorem]{Proposition}
\theoremstyle{definition}
\newtheorem{definition}[theorem]{Definition}
\newtheorem{remark}[theorem]{Remark}

\newcommand{\dd}{{\rm d}}

\newcommand{\N}{{\mathbb N}}
\newcommand{\R}{{\mathbb R}}
\newcommand{\ab}{[a,b]}

\numberwithin{equation}{section}

\begin{document}

\title{Extremal solutions of systems of measure differential equations and applications in the study of Stieltjes differential problems}
\author{
Rodrigo L\'opez Pouso\thanks{Department of Statistics, Mathematical Analysis and Optimization, University of Santiago de Compostela, Santiago de Compostela, Spain (email: {\tt rodrigo.lopez@usc.es})}, \,
Ignacio M\'arquez Alb\'es\thanks{Department of Statistics, Mathematical Analysis and Optimization, University of Santiago de Compostela, Santiago de Compostela, Spain (email: {\tt ignacio.marquez@rai.usc.es})}, \,
        Giselle A. Monteiro\thanks{Institute of Mathematics, Czech Academy of Sciences, Prague, Czech Republic 
        (email: {\tt gam@math.cas.cz})} \thanks{Mathematical Institute, Slovak Academy of Sciences, Bratislava,
         Slovakia}        }

\date{\today}
\maketitle

\begin{abstract}
We use lower and upper solutions to investigate the existence of the greatest and the least solutions for quasimonotone systems of measure differential equations. The established results are then used to study the solvability of Stieltjes differential equations; a recent unification of discrete, continuous and impulsive systems. The applicability of our results is illustrated in a simple model for bacteria population.


\end{abstract}

\section{Introduction}

The method of lower and upper solutions traces as far back as 1886, with Peano's work \cite{Peano}. Despite that, the existence of extremal solutions and their relation to lower and upper solutions continue to be the subject matter of many research papers on ordinary differential equations; for instance \cite{caotpo, Cid, lihela, Pouso4}. In recent years, special attention has been devoted to the question of solutions for discontinuous nonlinear differential equations; see e.g. \cite{Biles-Pouso, Biles-Schechter, HR, Pouso6}. In this regard, a few steps have been done towards the development of the corresponding theory of extremal solutions for measure differential equations \cite{as}. 

Measure differential equations, as introduced in \cite{fed}, are integral equations featuring the Kurzweil-Stieltjes integral. These equations are known to generalize other types of equations, such as classical differential equations, equations with impulses, or dynamic equations on time scales; see \cite{fed2, MS}. 

In this paper we are concerned with vectorial measure differential equations of the form
\begin{equation}\label{MDEint}
 \vec{y}(t) = \vec{y_0} + \int_{t_0}^{t} \vec{f}(s,\vec{y}(s))\,{\rm d}\vec g(s), \ \ \ t \in I,
\end{equation}
where $I=[t_0,t_0+L]$, $\vec{y_0}\in \R^n$, $\vec{f}:I\times \R^n\to\mathbb R^n$, $\vec g:I\to\mathbb R^n$
is nondecreasing and left-continuous. The integral on the right-hand side is to be understood as an integration process `component-by-component' in the Kurzweil--Stieltjes sense (see Section 3 for details). Accordingly, the equality \eqref{MDEint} corresponds to a system of measure differential equations. 

The objective of this paper is to establish the existence of the greatest and the least solutions for \eqref{MDEint} in the presence of a pair of well-ordered lower and upper solutions. For our purposes,  quasimonotonicity plays a key role. The main contribution of this paper lies on the fact that the nonlinear term $\vec f$ is assumed to satisfy weakened continuity hypotheses. Besides answering one of the questions posed in \cite{as}, our results provide an insight into the solvability of equations with functional arguments.

A secondary, but interesting, issue is the impact of our results in the theory of Stieltjes differential equations (also known as $g$--differential equations), \cite{fp}. As we will see, the integral equation (\ref{MDEint}) includes as a particular case $g$--differential systems like
\begin{equation}\label{Stieltint}
\vec x'_g(t)=\vec f(t,\vec x(t))\quad\mbox{for } g\mbox{-a.a. }t\in I,\quad \vec x(t_0)=\vec{x_0},
\end{equation}
where $\vec x'_g$ stands for the derivative with respect to a nondecreasing and left-continuous function $g$ (see \cite{pr} for definition and properties of the derivative). These equations have gain in popularity as they offer an unified approach for investigating discrete and continuous problems. Impulsive differential systems, for example, can be rewritten as (\ref{Stieltint}) by using an appropriate derivator $g$ with jump discontinuities at the times when impulses are prescribed. Benefiting from the relation between \eqref{MDEint} and \eqref{Stieltint}, we establish new theorems on extremal solutions for Stieltjes differential systems which somehow complement the study initiated in \cite{PM}. As an illustration of our results, we analyze the dynamics of a bacteria population using $g$--differential systems.

\medbreak

\section{Preliminaries}

In what follows we summarize some useful results concerning regulated functions. For details see e.g. \cite{F,H}.

Recall that a function is called regulated if the one sided-limits exist at all points of the domain. It is well-known that regulated functions are bounded and have at most countably many points of discontinuity. Given a regulated function $f:\ab\to\R^n$, we define
$f(a-)=f(a)$, $f(b+)=f(b)$ and
\[\Delta^+f(t)=f(t+)-f(t),\quad\quad \Delta^-f(t)=f(t)-f(t-),\quad t\in[a,b],\]
where $f(t+)$ and $f(t-)$ stand for the right-hand limit and the left-hand limit, respectively.

As usual, $G([a,b],\mathbb R^n)$ denotes the space of all regulated functions $f:[a,b]\to\mathbb R^n$, and it is a Banach space  when equipped with the supremum norm $\|f\|_\infty=\sup_{t\in[a,b]}\|f(t)\|$. In the case when $n=1$, we will write simply $G(\ab)$.

\begin{theorem} \label{regulated}
The following two statements are equivalent:
\begin{itemize}
  \item[{\rm\,(i)}] $f\in G(\ab).$
  \item[{\rm(ii)}] For every $\varepsilon\,{>}\,0$ there exists a division $D\,:\,a=\alpha_0<\alpha_1<\cdots<\alpha_{\nu(D)}=b$ 
such that for every $j\in\{1,\dots,\nu(D)\}$ and $s,t\in(\alpha_{j{-}1},\alpha_j)$, we have 
$|f(s)\,{-}\,f(t)|<\varepsilon$.
\end{itemize}
\end{theorem}

\begin{remark}\label{notation}
  Given a function $f\in G(\ab)$ and an arbitrary $\varepsilon>0$, we will denote by $D_{f,\varepsilon}$ the division whose existence is guaranteed by Theorem \ref{regulated}.
\end{remark}

Compactness in the space of regulated functions is connected with the following notion:

\begin{definition}
A set $\mathcal{A}\subset G([a,b])$ is said to be equiregulated if for every $\varepsilon>0$ and every $t_0\in [a,b]$ there exists $\delta>0$ such that:
\[
  |x(t)-x(t_0+)|<\varepsilon
  \quad\mbox{for all \ }  t_0<t<t_0+\delta,\quad x\in \mathcal{A},
\]
\[
  |x(t)-x(t_0-)|<\varepsilon
  \quad\mbox{for all \ }  t_0-\delta<t<t_0,\quad x\in \mathcal{A}.
\]
\end{definition}

In the lines of Theorem \ref{regulated}, we have the following characterization of equiregulated sets of functions.

\begin{lemma}\label{Lem}
The following statements are equivalent:
\begin{itemize}
  \item[{\rm\,(i)}] $\mathcal{A}{\subset}\,G(\ab)$ is equiregulated.
  \item[{\rm(ii)}] For every $\varepsilon\,{>}\,0$ there exists a division $D\,:\,a=\alpha_0<\alpha_1<\cdots<\alpha_{\nu(D)}=b$ 
  such that for every $f\in\mathcal{A},$ $j\in\{1,\dots,\nu(D)\}$ and
  $s,t\in(\alpha_{j{-}1},\alpha_j),$ we have $|f(s)\,{-}\,f(t)|<\varepsilon.$
\end{itemize}
\end{lemma}

Next theorem is the analogue of the Arzel\'a-Ascoli theorem in the space of regulated
functions.

\begin{theorem}\label{Fra}
A subset $\mathcal{A}$ of $G(\ab)$ is relatively compact  if and only if it is equiregulated and the set $\{f(t):f\in \mathcal{A} \}$ is bounded for each $t\in [a,b]$.
\end{theorem}

\begin{remark}\label{Fra1}
At this point is worth mentioning one particular case in which the assumptions ensuring compactness are satisfied. Let $\mathcal{A}\subset G(\ab)$ be such that there exist $M>0$ and nondecreasing function $h:[a,b] \to \mathbb R$ satisfying 
$$ | f(v) - f(u) | \leq h(v) - h(u)\quad \mbox{for \ }f \in \mathcal{A},\,[u,v]\subseteq[a,b],$$
and $|f(a)|\leq M$ for all $f \in \mathcal{A}$. By Lemma \ref{Lem} the set $\mathcal{A}$ is equiregulated, and 
obviously $\{f(t):f\in \mathcal{A} \}$ is bounded for each $t\in I$. Thus, in this case, $\mathcal{A}$ is relatively compact.
\end{remark}

The following result is concerned with pointwise supremum of regulated functions and, to the best of our knowledge, it is not available in the literature.

\begin{proposition}\label{sup}
Let $\mathcal{A}$ be a relatively compact subset of $G(\ab)$. Then, the function $\xi:\ab\to\R$ given by
\[
\xi(t)=\sup\{f(t):f\in \mathcal{A}\},\quad t\in\ab,
\]
is regulated.
\end{proposition}
\begin{proof}
  From Theorem \ref{Fra}, we know that $\{f(t):f\in \mathcal{A} \}$ is bounded for each $t\in\ab$; therefore, the function $\xi$ is well defined.
Moreover, since $\mathcal{A}$ is equiregulated, by Lemma \ref{Lem}, given $\varepsilon>0$ there exists
a division $D\,:\,a=\alpha_0<\alpha_1<\cdots<\alpha_{\nu(D)}=b$   
such that for $j\in\{1,\dots,\nu(D)\}$ we have
\[
|f(t)-f(s)|<\varepsilon\quad\mbox{for all \ }s,t\in(\alpha_{j{-}1},\alpha_j),\quad f\in\mathcal{A}.
\]
Fixed an arbitrary $j\in\{1,\dots,\nu(D)\}$, for $s,t\in(\alpha_{j{-}1},\alpha_j)$ we get
\[
f(s)-\varepsilon< f(t)<f(s)+\varepsilon\leq \xi(s)+\varepsilon\quad\mbox{for all \ } f\in\mathcal{A},
\]
and consequently
\[
\xi(s)-\varepsilon\leq\xi(t)\le \xi(s)+\varepsilon.
\]
In summary, $\xi$ satisfies assumption (ii) of Theorem \ref{regulated}, hence $\xi$ is regulated.
\end{proof}

In what follows, given functions $f:[a,b]\to \mathbb R^n$ and $g:[a,b]\to\mathbb R$, the Kurzweil-Stieltjes integral of $f$ with respect to $g$ on $[a,b]$ will be denoted by $\int_a^bf(s)\,\dd g(s)$, or simply $\int_a^bf\,\dd g$. Such an integral has the usual properties of linearity, additivity with respect to adjacent subintervals, as well as the properties to be presented next. The interested reader we refer
to \cite{STV} or \cite{Tv5}.

The following result is known as Hake property.

\begin{theorem}\label{Hake}
Let $f:[a,b]\to \mathbb R^n$ and $g:[a,b]\to\mathbb R$ be given.
\begin{enumerate}
	\item If
the integral $\int_t^bf\,\dd g$ exists for every $t\in (a,b]$, and $A\in\R^n$ is such that
$$\lim_{t\to a+}\left(\int_t^bf\,\dd g+f(a)(g(t)-g(a))\right)=A,$$
then $\int_a^bf\,\dd g$ exists and equals $A$.
	\item If
	the integral $\int_a^tf\,\dd g$ exists for every $t\in [a,b)$, and $A\in\R^n$ is such that
	$$\lim_{t\to b-}\left(\int_a^tf\,\dd g+f(b)(g(b)-g(t))\right)=A,$$
then $\int_a^bf\,\dd g$ exists and equals $A$.
\end{enumerate}
\end{theorem}

Next theorem summarizes some properties of the indefinite Kurzweil-Stieltjes integral.

\begin{theorem}\label{indefiniteKS}
Let $f:[a,b]\to \mathbb R^n$ and $g\in G(\ab)$ be such that
the integral $\int_a^bf\,\dd g$ exists. Then the function
\[
h(t)=\int_a^tf\,\dd g,\quad t\in[a,b],
\]
is regulated and satisfies
\begin{align*}
h(t+)&=h(t)+f(t)\Delta^+g(t),\quad t\in[a,b),\\
h(t-)&=h(t)-f(t)\Delta^-g(t),\quad t\in(a,b].
\end{align*}
\end{theorem}

\section{Vectorial measure differential equations}
Measure differential equations, in the sense introduced in \cite{fed}, are integral equations of the form
\[
y(t)=y_0+\int_{t_0}^{t}{f(s,y(s)) \, \dd g(s)}, \quad t \in I,
\]
where the integral is understood as the Kurzweil--Stieltjes integral with respect to a nondecreasing function 
$g:I\to\R$ and the function $f$ takes values in $\R^n$, $n\in\N$. Herein we propose a more general version of such an equation where not only $f$ can be a vectorial function but also the integrator $g$. More precisely, we are interested on equations
\begin{equation}\label{MDE}
 \vec{y}(t) = \vec{y_0} + \int_{t_0}^{t} \vec{f}(s,\vec{y}(s))\,{\rm d}\vec g(s), \ \ \ t \in I,
\end{equation}
where $I=[t_0,t_0+L]$, $\vec{y_0}\in \R^n$, $\vec{f}:I\times \R^n\to\mathbb R^n$ and $\vec g:I\to\mathbb R^n$. The integral on the right-hand side is simply a notation for a `component-by-component' integration process in the Kurzweil--Stieltjes sense, that is, writing
$$\vec{y}_0=(y_{0,1},\dots, y_{0,n}),\quad\vec{y}=(y_{1},\dots, y_{n}),\quad\vec f=(f_1,\dots,f_n),\quad\vec g=(g_1,\dots,g_n),$$
equation \eqref{MDE} corresponds to a systems of $n$ scalar equations, each of which reads as follows
\begin{equation}\label{MDEint2}
y_i(t) = y_{0,i} + \int_{t_0}^{t} f_i(s,\vec{y}(s))\,{\rm d} g_i(s), \ \ \ t \in I, \ \ \ i \in \{1,\dots,n\}.
\end{equation}
This interpretation of the integral in \eqref{MDE} is justified by the following vectorial equation
\begin{equation}\label{GLDE}
 \vec{y}(t) = \vec{y}_0 + \int_{t_0}^{t} \dd[G(s)]\,\vec{f}(s,\vec{y}(s)), \ \ \ t \in I,
\end{equation}
where for each $t\in I$, $G(t)\in\R^{n\times n}$ is the diagonal matrix
\[
G(t)=
\begin{bmatrix}
g_{1}(t) & 0 & \ldots & 0
\\
0 & g_{2}(t) & \ldots & 0
\\
\vdots & \vdots & \ddots & \vdots
\\
0 & 0 & \ldots & g_{n}(t)
\end{bmatrix}.
\]
In the case when $\vec{f}(t,\vec{y}(t))=\vec{y}(t)$, equation \eqref{GLDE} becomes a particular case of the
so-called generalized linear differential equation; a branch of Kurzweil equations theory which has been extensively
investigated in \cite{STV, Tv5}.

Clearly, by taking $g_i=g:I\to\R$ for all $i \in \{1,\dots,n\}$, from \eqref{MDE} we retrieve the notion of measure differential
equation introduced in \cite{fed}. To avoid any misunderstanding, equations of the type \eqref{MDE} will be called
vectorial measure differential equations. That said, we define the concept of solution for vectorial measure differential equations.

\begin{definition}
A function $\vec y\in G(I,\mathbb R^n),\, \vec y=(y_{1},\dots, y_{n}),$ is a solution of equation \eqref{MDE} if $y_i:I\to\mathbb R$ satisfies
\eqref{MDEint2} for each $i \in \{1,\dots,n\}$.
\end{definition}

In view of Theorem \ref{indefiniteKS}, we can see that a solution of \eqref{MDE}
somehow shares the discontinuity points of $\vec g$.

Next we introduce the key concepts related to the question of extremal solutions. For that, we consider a partial ordering in $\mathbb R^n$ as follows: given two vectors $\vec{x}=(x_1, \dots,x_n)$ and $\vec{y}=(y_1,\dots,y_n)$, we write $\vec{x} \le \vec{y}$ if $x_i \le y_i$ for each $i \in \{1,\dots,n\}$. Naturally, for functions $\vec{\alpha}, \, \vec{\beta}:I \to \mathbb R^n$, we write $\vec{\alpha} \le \vec{\beta}$ provided $\vec{\alpha}(t) \le \vec{\beta}(t)$ for every $t \in I$.
Moreover, if $\vec{\alpha},\,\vec{\beta}\in G(I,\R^n)$ are such that $\vec{\alpha} \le \vec{\beta}$, then we define the functional interval
\begin{equation*}\label{interval}
[\vec{\alpha},\vec{\beta}]=\{ \vec{\eta} \in G(I,\mathbb R^n) \, : \, \vec{\alpha} \le \vec{\eta} \le \vec{\beta}\}.
\end{equation*}
When we want to emphasize that we are considering (lower/upper) solutions which belong to a certain $[\vec{\alpha},\vec{\beta}]$, we say that $\vec\varphi$ is a (lower/upper) solution between $\vec{\alpha}$ and $\vec{\beta}$.

\medbreak

The extremal (greatest and least) solutions to vectorial measure differential equations are defined in the standard way considering the
aforementioned ordering, that is, if $\vec{y}: I\to \R^n$ is a solution of \eqref{MDE} we say that:

$\bullet$ $\vec{y}$ is the greatest solution of \eqref{MDE} on $I$ if any other solution $\vec{x}:I\to\R^n$
satisfies $\vec{x}\leq \vec{y}$;

$\bullet$ $\vec{y}$ is the least solution of \eqref{MDE} on $I$ if any other solution $\vec{x}:I\to\R^n$
satisfies $\vec{y}\leq \vec{x}$.

In the following sections we will investigate the existence of extremal solutions for \eqref{MDE} between lower and upper solutions.

\begin{definition}\label{lower}
A lower solution of (\ref{MDE}) is a function
$\vec{\alpha}\in G(I,\mathbb R^n)$ such that $\vec{\alpha}(t_0)\leq \vec{y_0}$ and 
$$\alpha_i(v)-\alpha_i(u) \le \int_{u}^{v}f_i(s,\vec{\alpha}(s)) \, \dd g_i(s),\quad [u,v]\subseteq I,
\quad i \in \{1,\dots,n\}.$$
Symmetrically, an upper solution of (\ref{MDE}) is a function
$\vec{\beta}\in G(I,\mathbb R^n)$ such that $\vec{y}_0\leq\vec{\beta}(t_0)$ and
$$\beta_i(v)-\beta_i(u) \geq \int_{u}^{v}f_i(s,\vec{\beta}(s)) \, \dd g_i(s),\quad [u,v]\subseteq I,
\quad i \in \{1,\dots,n\}.$$
\end{definition}

\begin{remark}\label{lower+}
If $\vec{\alpha}:I\to \R^n$ is a lower solution of \eqref{MDE}, then for all $i\in\{1,\dots, n\}$
\begin{align}
\Delta^+\alpha_i(t)&=\alpha_i(t+)-\alpha_i(t)\leq f_i(t,\vec{\alpha}(t))\Delta^+g_i(t), \quad t\in I,\label{lower-jump1}\\
\Delta^-\alpha_i(t)&=\alpha_i(t)-\alpha_i(t-)\leq f_i(t,\vec{\alpha}(t))\Delta^-g_i(t), \quad t\in I.\label{lower-jump2}
 \end{align}
Obviously, the reverse inequalities hold for upper solutions.
\end{remark}

\section{An existence result for the scalar case}
In this section, we turn our attention to the scalar case of equation (\ref{MDE}), that is, equations of the form
\begin{equation}\label{MDE1}
x(t)=x_0+\int_{t_0}^{t}{f(s,x(s)) \, \dd g(s)}, \quad t \in I,
\end{equation}
where $x_0\in\R$, $f:I \times \mathbb R \to \mathbb R$ and $g:I\to\mathbb R$ is nondecreasing and left-continuous.

\smallskip
The existence of the greatest and the least solutions for \eqref{MDE1} has been already investigated in \cite{as}.
In this section, we address one of the questions posed in \cite{as}, namely, the existence of (extremal) solutions between
given lower and upper solutions. Our result somehow generalizes what is available in the classical theory of ordinary differential
equations (cf. \cite{lp}) as the function $f$ is not required to be continuous with respect to the first variable.

For the convenience of the reader we will recall the main results in \cite{as}. Given a set $B\subseteq\R$, we consider the following conditions:
\begin{enumerate}
\item[{\rm(C1)}] The integral $\int_{t_0}^{t_0+L} f(t,y)\,{\rm d}g(t)$ exists for every $y\in B$.
\item[{\rm(C2)}] There exists a function $M:I\to\mathbb R$, which is Kurzweil-Stieltjes integrable with respect to $g$, such that
\[
\left|\int_u^v f(t,y)\,{\rm d}g(t)\right|\leq \int_u^v M(t)\,{\rm d}g(t)
\]
for every $y\in B$ and $[u,v]\subseteq I$.	
\item[{\rm (C3)}] For each $t\in I$, the mapping \ $y\mapsto f(t,y)$ is continuous in $B$.
\end{enumerate}

Next lemma is a important tool for dealing with conditions above (see \cite[Lemma 3.1]{as}).

\begin{lemma}\label{conditions-regulated}
Assume that $f:I \times B\to\mathbb R$ satisfies conditions {\rm (C1), (C2), (C3)}. Then for each function $x\in G(I,B)$, the integral $\int_{t_0}^{t_0+L} f(t,x(t))\,\dd g(t)$ exists, and we have
\begin{equation}\label{estimate}
\left|\int_u^v f(t,x(t))\,{\rm d}g(t)\right|\leq \int_u^v M(t)\,{\rm d}g(t),\quad [u,v]\subseteq I.
\end{equation}
\end{lemma}

The following existence result for equations \eqref{MDE1} derives from  \cite[Theorem 3.2]{as}.

\begin{theorem}\label{MDE-existence}
If $f:I\times \R\to\mathbb R$ satisfies conditions {\rm (C1)}, {\rm (C2)}, {\rm (C3)} with $B=\R$, then equation \eqref{MDE1} has a solution on $I$.
\end{theorem}

Based on Theorems 4.4 and 4.12 from \cite{as}, we deduce the following result about extremal solutions.

\begin{theorem}\label{MDE-greatest}
Suppose that $f:I\times \R\to\mathbb R$ satisfies conditions {\rm (C1)}, {\rm (C2)}, {\rm (C3)} with $B=\R$. Further, assume that
\begin{enumerate}
  \item[{\rm (C4)}] for each $t\in I$ the mapping \ $u\in\R\mapsto u+f(t,u)\Delta^+g(t)$ is nondecreasing.
\end{enumerate}
If equation \eqref{MDE1} has a solution on $I$, then it has the greatest solution $x^*$ and the least solution $x_*$ on $I$.
Moreover, for each $t\in I$ we have
\begin{align*}
	x^*(t)&=\sup\{\alpha(t):\,\alpha\mbox{ is a lower solution of \eqref{MDE1} on }[t_0,t]\},\\
	x_*(t)&=\inf\{\beta(t):\,\beta\mbox{ is an upper solution of \eqref{MDE1} on }[t_0,t]\}.
\end{align*}
\end{theorem}

Now, we will investigate the existence of extremal solutions for equation \eqref{MDE1} provided a lower and an upper solutions are known and well-ordered.

\begin{theorem}
\label{exn1}
Suppose that \eqref{MDE1} has a lower solution $\alpha$ and an upper solution $\beta$ such that $\alpha(t) \leq \beta(t)$ for all $t\in I$.
Assume that the following conditions hold:
\begin{enumerate}
\item[{\rm(H1)}] The integral $\int_{t_0}^{t_0+L} f(t,y)\,{\rm d}g(t)$ exists for every 
$y \in E=[\inf_{s\in I} \alpha(s)\,,\,\sup_{s\in I}\beta(s)]$.
\item[{\rm(H2)}] There exists a function $M:I\to\mathbb R$, which is Kurzweil-Stieltjes integrable with respect to $g$, such that
\[
\left|\int_u^v f(t, y)\,{\rm d}g(t)\right|\leq \int_u^v M(t)\,{\rm d}g(t)
\]
for every $y\in E$ and $[u,v]\subseteq I$.
\item[{\rm(H3)}] For each $t\in I$, the mapping $y\mapsto f(t,y)$ is continuous in $E$.
\item[{\rm(H4)}] For each $t\in I$ the mapping
\[
u\in[\alpha(t),\beta(t)]\mapsto u+f(t,u)\Delta^+g(t)
\]
is nondecreasing.
\end{enumerate}
Then equation \eqref{MDE1} has extremal solutions between $\alpha$ and $\beta$. 
Moreover, for each $t\in I$ we have
\begin{equation}
\label{x-greatest}
x^*(t)=\sup \{\ell(t) \, : \, \mbox{$\ell$ lower solution of \textup{(\ref{MDE1})} between $\alpha$ and $\beta$} \},
\end{equation}
\begin{equation}
\label{x-least}
x_{*}(t)=\inf \{u(t) \, : \, \mbox{$u$ upper solution of {\rm (\ref{MDE1})} between $\alpha$ and $\beta$} \}.
\end{equation}
\end{theorem}
\begin{proof}
Let us define $\widetilde f:I\times\mathbb R\rightarrow\mathbb R$ as
\[
\widetilde f(t,x)=\left\{
\begin{array}{cl}
f(t, \alpha(t))&\mbox{if $x< \alpha(t)$,}\\
f(t,x)&\mbox{if $\alpha(t) \le x \le \beta(t),$}\\
f(t,\beta(t))&\mbox{if $x>\beta(t)$,}
\end{array}
\right.
\]
and consider the modified problem
\begin{equation}\label{prima1}
x(t)=x_0+\int_{t_0}^{t}{\widetilde f(s,x(s)) \, \dd g(s)}, \quad t\in I.
\end{equation}
Clearly, (H3) ensures that $\widetilde f$ satisfies (C3) with $B=\R$. To show that $\widetilde f$ satisfies both conditions (C1) and (C2) with $B=\R$, let $y\in \R$ and put $m(t)=\max\{\min\{y,\beta(t)\},\alpha(t)\}$, $t\in I$. Thus, $m\in G(I,E)$ and $\widetilde{f}(t,y)=f(t,m(t))$ for every $t\in I$. Note that Lemma \ref{conditions-regulated} and conditions (H1)--(H3) imply that
$\int_{t_0}^{t_0+L} f(t,m(t))\,{\rm d}g(t)$ exists 
and \eqref{estimate} hold; in other words, for $y\in \R$, the integral $\int_{t_0}^{t_0+L}\widetilde f(t,y)\,{\rm d}g(t)$ exists and 
\[
\left|\int_u^v\widetilde f(t, y)\,{\rm d}g(t)\right|\leq \int_u^v M(t)\,{\rm d}g(t),\quad [u,v]\subseteq I.
\]
In summary, $\widetilde f$ satisfies the conditions of Theorem \ref{MDE-existence} and
we conclude that (\ref{prima1}) has a solution defined on the whole of $I$. The existence of the greatest solution $x^*$ and the least
solution $x_*$ of \eqref{prima1} is then a consequence of Theorem \ref{MDE-greatest} and assumption (H4).

It only remains to show that if $x:I\to\R$ is an arbitrary solution of equation (\ref{prima1}), then $\alpha(t) \le x(t) \le \beta(t)$, $t\in I$,
thus proving that $x$ is a solution of (\ref{MDE1}) and the functions $x_*$ and $x^*$ are the intended extremal solutions between $\alpha$ and $\beta$. Reasoning by contradiction, assume that there exists some $t_1\in(t_0,t_0+L]$ such that
\begin{equation}\label{t_1}
  \alpha(t_1)>x(t_1).
\end{equation}
Let $t_2=\sup\{t\in[t_0,t_1):\,\alpha(t)\le x(t)\}.$
By the definition of supremum, either $\alpha(t_2)\leq x(t_2)$ (which includes the case $t_2=t_0$) or there exists a sequence of points $u_k\in (t_0,t_2),k\in\N,$ such that $u_k\to t_2$ and $\alpha(u_k)\le x(u_k)$ for each $k\in\mathbb N$. Therefore, $\alpha(t_2-)\le x(t_2-)$.
Using \eqref{lower-jump2} and the fact that $g$ is left-continuous, we get $\Delta^-\alpha(t_2)\le 0$, that is,
\[
\alpha(t_2)\leq\alpha(t_2-)\le x(t_2-)=x(t_2).
\]
Hence, we must have $\alpha(t_2)\le x(t_2)$ and consequently $t_2<t_1$. We will show that this leads to a contradiction with \eqref{t_1}.
First, observe that
\begin{align}
  \alpha(t_1)-x(t_1)&= \alpha(t_1)-\alpha(t_2)+\alpha(t_2)-x(t_1)\nonumber\\
  &\leq \int_{t_2}^{t_1}f(s,\alpha(s)) \, \dd g(s)+\alpha(t_2)-x(t_1). \label{c}
\end{align}
The definition of $t_2$ implies that $x(t)<\alpha(t)$, $t\in(t_2,t_1]$; thus, by Theorem \ref{Hake} we have
\begin{align*}
   \int_{t_2}^{t_1}f(s,\alpha(s)) \, \dd g(s)  &= \lim_{\sigma\to t_2+} \int_{\sigma}^{t_1}f(s,\alpha(s)) \, \dd g(s) +f(t_2,\alpha(t_2))\Delta^+g(t_2) \\
   & = \lim_{\sigma\to t_2+} \int_{\sigma}^{t_1}\widetilde{f}(s,x(s)) \, \dd g(s) +f(t_2,\alpha(t_2))\Delta^+g(t_2)\\
   &  = \int_{t_2}^{t_1}\widetilde{f}(s,x(s)) \, \dd g(s) -\widetilde{f}(t_2,x(t_2))\Delta^+g(t_2)+f(t_2,\alpha(t_2))\Delta^+g(t_2)\\
   &  = x(t_1)-x(t_2)-\widetilde{f}(t_2,x(t_2))\Delta^+g(t_2)+f(t_2,\alpha(t_2))\Delta^+g(t_2).
\end{align*}
Combining this equality with \eqref{c} we obtain
$$\alpha(t_1)-x(t_1) \leq \alpha(t_2)+f(t_2,\alpha(t_2))\Delta^+g(t_2)-x(t_2)-\widetilde{f}(t_2,x(t_2))\Delta^+g(t_2).$$
At this point we need to distinguish two cases regarding the value of $\widetilde{f}$. If $\alpha(t_2)\leq x(t_2)\leq \beta(t_2)$, then
$\widetilde{f}(t_2,x(t_2))= f(t_2,x(t_2))$. Since condition (H4) implies that
\[
\alpha(t_2)+f(t,\alpha(t_2))\Delta^+g(t_2)\le x(t_2)+f(t,x(t_2))\Delta^+g(t_2),
\]
we conclude that $\alpha(t_1)-x(t_1)\leq 0$, a contradiction. In the case when $x(t_2)> \beta(t_2)$, then $\widetilde{f}(t_2,x(t_2))= f(t_2,\beta(t_2))$, which implies that
$$
  \alpha(t_1)-x(t_1) \leq \alpha(t_2)+f(t_2,\alpha(t_2))\Delta^+g(t_2)-\beta(t_2)-f(t_2,\beta(t_2))\Delta^+g(t_2).
$$
The contradiction again follows from condition (H4), now taking into account $\alpha(t_2)\leq \beta(t_2)$.

The proof that $x \le \beta$ on $I$ is analogous and we omit it. Moreover, equalities \eqref{x-greatest} and \eqref{x-least} follow from Theorem \ref{MDE-greatest}.
\end{proof}

\section{Extremal solutions for vectorial measure differential equations}
Our goal is to extend the results from Section 4 to the vectorial equation
\begin{equation}\label{MDE-v}
 \vec{y}(t) = \vec{y_0} + \int_{t_0}^{t} \vec{f}(s,\vec{y}(s))\,{\rm d}\vec g(s), \ \ \ t \in I,
\end{equation}
where $\vec{y_0}\in \R^n$, $\vec{f}:I\times \R^n\to\mathbb R^n$ and $\vec g:I\to\mathbb R^n$.
Herein, we assume that $\vec g=(g_1,\dots,g_n)$ is nondecreasing and left-continuous, that is, 
for each $i\in\{1,\dots, n\}$, the function $g_i:I\to\mathbb R$ is nondecreasing and left-continuous. 

Like in the theory of
multidimensional equations, the function $\vec f$ is required to be quasimonotone. 
Recall that a function $\vec{f}$ is quasimonotone nondecreasing in a set $E\subseteq I\times \R^n$ if 
given $t\in I$ and vectors $\vec{x}=(x_1,\dots,x_n),$ $\vec y=(y_1,\dots, y_n)$ such that $(t,\vec x),(t,\vec y)\in E$, the following holds:
$$\vec x\leq\vec y\mbox{ \ with $x_i=y_i$ for some }i\in\{1,\dots,n\}
	\quad\Longrightarrow\quad f_i(t,\vec x)\leq f_i(t,\vec y).$$

In what follows, $\vec{e_i}$, $i\in\{1,\dots, n\}$, denotes the vector in $\R^n$ whose $i$-th term is $1$ and all others are zero.

\begin{theorem}\label{extrMDE}
Suppose that \eqref{MDE-v} has a lower solution $\vec\alpha$ and an upper solution $\vec\beta$ such that $\vec{\alpha}\leq\vec{\beta}$ and assume that $\vec{f}$ is quasimonotone nondecreasing in
$E=\{(t,\vec{x})\in I\times\R^n\,:\,\vec{\alpha}(t)\leq\vec{x}\leq\vec{\beta}(t)\}$.
Furthermore, assume that the following conditions hold:
\begin{enumerate}
\item[$(\mathcal H1)$] The integral $\int_{t_0}^{t_0+L} f_i(t,\vec{\eta}(t))\,{\rm d}g_i(t)$ exists for every $\vec{\eta}\in[\vec{\alpha},\vec{\beta}]$ and $i\in\{1,\dots, n\}$.
\item[$(\mathcal H2)$] For each $i\in\{1,\dots, n\}$, there exists a function $M_i:I\to\mathbb R$, which is Kurzweil-Stieltjes integrable with respect to $g_i$, such that
\[
\left|\int_u^v f_i(t,\vec{\eta}(t))\,{\rm d}g_i(t)\right|\leq \int_u^v M_i(t)\,{\rm d}g_i(t)
\]
for every $\vec{\eta}\in[\vec{\alpha},\vec{\beta}]$ and $[u,v]\subseteq I$.
\item[$(\mathcal H3)$] For each $\vec{\eta}\in[\vec{\alpha}, \vec{\beta}]$, $i\in\{1,\dots,n\}$, and $t\in I$, the mapping
\[
u\in[\alpha_i(t),\beta_i(t)]\mapsto  f_i(t,\vec{\eta}(t)+(u-\eta_i(t))\vec{e_i})
\]
is  continuous.
\item[$(\mathcal H4)$] For each $\vec{\eta}\in[\vec{\alpha}, \vec{\beta}]$, $i\in\{1,\dots,n\}$, and $t\in I$, the mapping
\[
u\in[\alpha_i(t),\beta_i(t)]\mapsto u+f_i(t,\vec{\eta}(t)+(u-\eta_i(t))\vec{e_i})\Delta^+g_i(t)
\]
is nondecreasing.
\end{enumerate}
Then equation \eqref{MDE-v} has extremal solutions in $[\vec{\alpha},\vec{\beta}]$. Moreover, for $t\in I$, the greatest solution
$\vec{y}^*=(y^*_1,\dots,y^*_n)$ is given by
\begin{equation}
\label{xmaxsys}
y^*_i(t)=\sup \{\ell_i(t) \, : \, \mbox{$(\ell_1,\dots,\ell_n)$ lower solution of \eqref{MDE-v} in $[\vec{\alpha},\vec{\beta}]$} \},
\end{equation}
and the least solution $\vec{y}_*=(y_{*,1},\dots,y_{*,n})$ is given by
\begin{equation}
\label{xminsys}
y_{*,i}(t)=\inf \{u_i(t) \, : \, \mbox{$(u_1,\dots,u_n)$ upper solution of \eqref{MDE-v} in $[\vec{\alpha},\vec{\beta}]$} \}.
\end{equation}
\end{theorem}
\begin{proof}
For $i\in\{1,\dots,n\}$ and $t\in I$, put $h_i(t)=\int_{t_0}^tM_i(s)\,\dd g_i(s)$ where $M_i$ is the function in $(\mathcal H2)$. Hence each function $h_i:I\to\R$, $i\in\{1,\dots,n\}$, is nondecreasing and left-continuous. Let $\vec{L}=(L_1,\dots,L_n)$ be an arbitrary lower solution of \eqref{MDE-v} in $[\vec{\alpha},\vec{\beta}]$ and
consider the set $\mathcal{A}$ of functions $\vec{\eta}\in [\vec{\alpha}, \vec{\beta}]$, $\vec{\eta}=(\eta_1,\dots,\eta_n)$,
satisfying the following two conditions:
\begin{align}\label{2.}
&\vec{\eta}(t_0)\leq \vec{y_0}\mbox{ \ and \ }
	\eta_i(v)-\eta_i(u)\leq \int_u^v M_i(s)\,{\rm d}g_i(s)\quad\mbox{for \ }[u,v]\subseteq I,\quad i\in\{1,\dots,n\};
\\\label{3.}
&\mbox{for each  \ }i\in\{1,\dots,n\}\mbox{ \ and each \ }\varepsilon>0,\,\,
	D_{\eta_i,\varepsilon}\subset D_{L_i,\varepsilon}\cup D_{h_i,\varepsilon}.
\end{align}
It is not hard to see that $\vec{L}\in\mathcal{A}$; moreover, every solution of \eqref{MDE-v} in
$[\vec{\alpha}, \vec{\beta}]$ belongs to $\mathcal{A}$ (in such a case, condition \eqref{3.} is a consequence of $(\mathcal H2)$).

Define $\vec{\xi^*}=(\xi_1^*,\dots,\xi_n^*)$ where, for each $i\in\{1,\dots,n\}$, \ $\xi_i^*:I\to\R$ is the function given by
\begin{equation}\label{xmaxsys2}
\xi_i^*(t)=\sup\{\eta_i(t): \vec{\eta}\in \mathcal{A} \mbox{ \ and \ $\vec{\eta}$ is a lower solution}\},
\quad t\in I.
\end{equation}
Note that, for each $t\in I$ and $i\in\{1,\dots,n\}$, the set $\{\eta_i(t): \vec{\eta}\in \mathcal{A}\}\subset [\alpha_i(t),\beta_i(t)]$. Therefore, the supremum $\xi_i^*(t)$ is well-defined. Moreover, condition \eqref{3.} ensures that  $\{\eta_i: \vec{\eta}\in \mathcal{A}\}$
is equiregulated (see Lemma \ref{Lem}). Thus, Theorem \ref{Fra} together with Proposition \ref{sup} 
implies that $\xi_i^*$ is regulated for each
$i\in\{1,\dots,n\}$, and consequently $\vec\xi^*\in G(I,\R^n)$.

\medskip

\noindent
{\it Claim 1 -- $\vec{\xi^*}$ is the greatest solution of \textup{(\ref{MDE-v})} in $[\vec{\alpha},\vec{\beta}]$.}
\\
Fix an arbitrary $i\in\{1,\dots,n\}$ and define the function $\Phi_i:I\times\mathbb R\rightarrow\mathbb R$ by $$\Phi_i(t,x)=f_i(t,\vec{\xi^*}(t)+(x-\xi^*_i(t))\vec{e_i}),\quad t\in I,\quad x\in\R.$$
Since $\vec \xi \leq\vec \beta$, the quasimonotonicity of $\vec{f}$ yields
$$\int_{u}^{v}\Phi_i(s,\beta_i(s)) \,\dd g_i(s)
	\leq\int_{u}^{v}f_i(s,\vec{\beta}(s))\,\dd g_i(s)\le \beta_i(v)-\beta(u),
\quad [u,v]\subseteq I,$$
which shows that $\beta_i:I\to\R$ is an upper solution of the scalar problem
\begin{equation}
\label{auxsys}
x(t)=y_{0,i}+\int_{t_0}^{t}\Phi_i(s,x(s)) \,\dd g_i(s), \quad t \in I.
\end{equation}
Using a similar argument, we can show that for any lower solution $\vec{\eta}$ of \eqref{MDE-v} such that 
$\vec{\eta} \in \mathcal{A}$,
the function $\eta_i:I\to\R$ is a lower solution of \eqref{auxsys} between $\alpha_i$ and $\beta_i$. 
Noting that $\Phi_i$ satisfies the conditions of Theorem \ref{exn1}, it follows that (\ref{auxsys}) has the greatest solution $x_i^*:I\to\R$
between $\alpha_i$ and $\beta_i$, and by \eqref{x-greatest} $\eta_i(t) \le x_i^*(t)$, $t\in I$, for any $\vec{\eta} \in \mathcal{A}$ lower solution of \eqref{MDE-v}.
Since the argument is valid for each $i\in\{1,\dots,n\}$, we construct a function $\vec{x^*}=(x_1^*,\dots,x_n^*)$, and obviously
$\vec{\xi^*} \le \vec{x^*}$. The quasimonotonicity of $\vec f$ yields 
$$x_i^*(v)-x_i^*(u)=\int_{u}^{v}\Phi_i(s,x_i^*(s)) \, \dd g_i(s) \le \int_{u}^{v}f_i(s,\vec{x^*}(s)) \,\dd g_i(s),
\quad [u,v]\subseteq I,$$
for each $i\in\{1,\dots,n\}$, that is, $\vec{x^*}$ is a lower solution of (\ref{MDE-v}) in $[\vec{\alpha},\vec{\beta}]$, and
\begin{eqnarray}
|x_i^*(v)-x_i^*(u)|&=&\left|\int_{u}^{v} \Phi_i(s,x_i^*(s))\, \dd g_i(s)\right|
	=\left|\int_{u}^{v}f_i(s,\xi^*(s)+(x_i^*(s)-\xi_i^*(s))\vec e_i)\, \dd g_i(s)\right|
	\nonumber\\
&\leq& \int_{u}^{v}M_i(s)\, \dd g_i(s)=h_i(v)-h_i(u)\nonumber
\end{eqnarray}
for every $[u,v]\subseteq I$ and $i\in\{1,\dots,n\}$. This shows that $\vec x^*$ satisfies \eqref{2.} and \eqref{3.}.
Thus $\vec{x^*} \in \mathcal{A}$, 
and the definition of $\vec{\xi^*}$ implies $\vec{x^*}\leq \vec{\xi^*}.$ In summary, $\vec{\xi}^*=\vec{x^*}$ and
$$\xi_i^*(t)=y_{0,i}+\int_{t_0}^{t}\Phi_i(s,\xi_i^*(s)) \, \dd g_i(s)=y_{0,i}+\int_{t_0}^{t}f_i(s,\vec{\xi^*}(s)) \,\dd g_i(s),
    \quad t\in I,\,i\in\{1,\dots,n\}.$$
Therefore, $\vec{\xi^*}$ is a  solution of (\ref{MDE-v}), and, by (\ref{xmaxsys2}) 
it is the greatest one in $[\vec{\alpha},\vec{\beta}]$.

\medskip

\noindent
{\it Claim 2 -- The greatest solution of \textup{(\ref{MDE-v})} in $[\vec{\alpha},\vec{\beta}]$, $\vec{y^*}=\vec{\xi^*}$, satisfies \textup{(\ref{xmaxsys})}.}
\\
The lower solution $\vec{L} \in [\vec{\alpha},\vec{\beta}]$ was fixed arbitrarily, so $\vec{y^*}$ is greater than or equal to any lower solution
in $[\vec{\alpha},\vec{\beta}]$. On the other hand, $\vec{y^*}$ is a lower solution itself and so (\ref{xmaxsys}) holds.

\medskip
The proof of the existence of the least solution $\vec{y_*}$ as well the validity of (\ref{xminsys}) is analogous.
\end{proof}

\section{Extremal solutions for vectorial measure differential equations with functional arguments}
We will now consider the functional problem
\begin{equation}\label{FMDE}
 \vec{y}(t) = \vec{y_0} + \int_{t_0}^{t} \vec{f}(s,\vec{y}(s),\vec y)\,{\rm d}\vec g(s), \ \ \ t \in I,
\end{equation}
where $\vec{y_0}\in \R^n$, $\vec{f}:I\times \R^n\times G(I,\R^n)\to\mathbb R^n$ and $\vec g:I\to\mathbb R^n$ is nondecreasing and left-continuous. We recall that the integral on the right-hand side should be understood as a vectorial Kurzweil-Stieltjes in the sense presented in Section 3.  

\

Equations \eqref{FMDE} subjected to functional arguments represent a quite general object.
It is not hard to see that functional differential equations of the form
\begin{equation*}\label{func-eq}
\vec y\,'(t)=\vec f(t,\vec y(t),\vec y)
\end{equation*}
can be regarded as \eqref{FMDE} provided the integral of $\vec f$ exists in some sense  
(in such a case $g_i$ corresponds to the identity function for each $i$). 
The class of problems covered by \eqref{FMDE} also includes the so-called
measure functional differential equations in the sense introduced in \cite{fed}.
To see this it is enough to consider
\[
\vec{f}(t,\vec{y}(t),\vec y)=\vec F(t,\vec y_t)\quad\mbox{and}\quad \vec g=(g,\dots,g),
\]
where $r>0$, $\vec F:I\times G([-r,0],\R^n)\to\mathbb R^n$, $g:I\to\mathbb R$ is nondecreasing 
and left-continuous, and for each $t\in I$ the function $\vec y_t:[-r,0]\to\R^n$ denotes the history or 
memory of $\vec y$ in $[t-r,t]$, that is, $\vec{y}_t(\theta)=\vec{y}(t+\theta),$ $\theta\in [-r,0].$

\

Unlike the work developed in previous sections, to investigate the extremal solutions for the problem
\eqref{FMDE} we will use a fixed-point approach; namely, the following result which is a consequence of 
\cite[Theorem 1.2.2]{lihela}.

\begin{proposition} \label{fixedpoint}
Let $\vec\alpha,\vec\beta\in G(I,\R^n)$ be such that $\vec\alpha\leq\vec\beta$ and let $T:[\vec\alpha,\vec\beta]\rightarrow[\vec\alpha,\vec\beta]$ be a nondecreasing map. 
Assume that for each $i\in\{1,\dots, n\}$ there exists a nondecreasing map $h_i:I\rightarrow\R$ such that for all $\vec\gamma\in[\vec\alpha,\vec\beta]$ and all $[u,v]\subseteq I$ the following inequality holds
\begin{equation}\label{Tbound}
|(T\vec\gamma)_i(v)-(T\vec\gamma)_i(u)|\leq h_i(v)-h_i(u).
\end{equation}
Then $T$ has the least fixed point $\vec\gamma_*$ and the greatest fixed point $\vec\gamma^*$ 
in $[\vec\alpha,\vec\beta]$. Moreover,
$$\vec\gamma_*=\min\{\vec\gamma\in[\vec\alpha,\vec\beta]:T\vec\gamma\leq \vec\gamma\},\quad
\vec\gamma^*=\max\{\vec\gamma\in[\vec\alpha,\vec\beta]:\vec\gamma\leq T\vec\gamma\}.$$
\end{proposition}
\begin{proof}
We will apply \cite[Theorem 1.2.2]{lihela} assuming $X=Y=G(I,\R^n)$ equipped with the supremum norm and the partial ordering defined in Section 3. 
Given a monotone sequence $\{\vec\gamma_k\}_{k=0}^\infty$ in $[\vec\alpha,\vec\beta],$ it suffices to show that for each 
$i\in\{1,\dots,n\}$ the sequence $\{(T\vec\gamma_k)_i\}_{k=0}^\infty$ converges in $G(I)$. 
By \eqref{Tbound} and Remark \ref{Fra1}, for each $i\in\{1,\dots,n\}$, $\{(T\vec\gamma_k)_i\}_{k=0}^\infty$ is a relatively compact subset of $G(I)$, 
hence it contains a convergent subsequence. The result then follows from the monotonicity of the sequence 
$\{(T\vec\gamma_k)_i\}_{k=0}^\infty$.
\end{proof}

Next result is the analogue of Theorem \ref{extrMDE} for functional equations. 
Note that the notion of lower and upper solutions for equation \eqref{FMDE} is an obvious extension 
of Definition \ref{lower}. Indeed, $\vec{\alpha}\in G(I,\mathbb R^n)$ is a lower solution of \eqref{FMDE} 
provided $\vec{\alpha}(t_0)\leq \vec{y_0}$ and 
$$\alpha_i(v)-\alpha_i(u) \le \int_{u}^{v}f_i(s,\vec{\alpha}(s),\vec\alpha)\,\dd g_i(s),\quad [u,v]\subseteq I,\quad i\in\{1,\dots, n\},$$ 
while the reverse inequalities are used to define upper solutions of \eqref{FMDE}.

\begin{theorem}\label{funcexist}
Suppose that \eqref{FMDE} has a lower solution $\vec\alpha$ and an upper solution $\vec\beta$ such that $\vec{\alpha}\leq\vec{\beta}$. 
For each $\vec\gamma\in[\vec\alpha,\vec\beta]$, denote by 
$\vec f_{\vec\gamma}:I\times\mathbb R^n\rightarrow\mathbb R^n$ the function defined as  
$\vec f_{\vec\gamma}(t,\vec x)=\vec f(t,\vec x,\vec\gamma)$.  
Assume that for each $\vec\gamma\in[\vec\alpha,\vec\beta]$, the function $\vec f_{\vec\gamma}$ is quasimonotone 
nondecreasing in $E=\{(t,\vec{x})\in I\times\R^n\,:\,\vec{\alpha}(t)\leq\vec{x}\leq\vec{\beta}(t)\}$. Furthermore, assume that the following conditions hold:
\begin{enumerate}
\item[$(\mathbb{H}1)$] 
The integral $\int_{t_0}^{t_0+L} (f_{\vec\gamma})_i(s,\vec\eta(t))\, \dd g_i(s)$ exists for every $\vec\gamma,\vec\eta\in[\vec\alpha,\vec\beta]$ and $i\in\{1,\dots, n\}$.
\item[$(\mathbb{H}2)$] 
For each $i\in\{1,\dots, n\}$, there exists $M_i:I\to \R,$ which is Kurzweil-Stieltjes integrable with respect to $g_i$, such that
\[
\left|\int_u^v (f_{\vec{\gamma}})_i(t,\vec{\eta}(t))\,{\rm d}g_i(t)\right|\leq \int_u^v M_i(t)\,{\rm d}g_i(t)
\]
for every $\vec\gamma, \vec{\eta}\in[\vec{\alpha},\vec{\beta}]$ and $[u,v]\subseteq I$.
\item[$(\mathbb{H}3)$] 
For each $\vec\gamma, \vec\eta\in[\vec\alpha,\vec\beta]$, $i\in\{1,...,n\},$ and $t\in I$, the mapping
$$u\in[\alpha_i(t),\beta_i(t)]\mapsto  (f_{\vec\gamma})_i(t,\vec\eta(t)+(u-\eta_i(t))\vec{e_i}) $$
is continuous.
\item[$(\mathbb{H}4)$] 
For each $\vec\gamma,\vec\eta\in[\vec\alpha,\vec\beta]$, $i\in\{1,...,n\},$ and $t\in I$, the mapping
$$u\in[\alpha_i(t),\beta_i(t)]\mapsto u+(f_{\vec\gamma})_i(t,\vec\eta(t)+(u-\eta_i(t))\vec{e_i})\,\Delta^+g_i(t)$$
is nondecreasing.
\item[$(\mathbb{H}5)$]  
For each $t\in[t_0,t_0+L)$ and $\vec x\in\mathbb R^n,$ the mapping $\vec f(t,\vec x,\cdot)$ is nondecreasing on $[\vec\alpha,\vec\beta].$
\end{enumerate}
Then equation \eqref{FMDE} has extremal solutions in $[\vec\alpha,\vec\beta].$
\end{theorem}
\begin{proof}
Note that, for each $\vec\gamma\in[\vec\alpha,\vec\beta],$ assumption $(\mathbb{H}5)$ implies that $\vec\alpha$ and $\vec\beta$, respectively, are lower and upper solutions of the vectorial equation 
\begin{equation}\label{FMDE2}
\vec y(t)=\vec{y_0}+\int_{t_0}^t \vec{f}_{\vec{\gamma}}(s,\vec y(s))\, \dd \vec g(s).
\end{equation}

Consider the map $T:[\vec\alpha,\vec\beta]\rightarrow[\vec\alpha,\vec\beta]$ defined as follows: for each 
$\vec\gamma\in[\vec\alpha,\vec\beta],$ $T\vec\gamma$ is the greatest solution of \eqref{FMDE2} in 
$[\vec\alpha,\vec\beta]$.
The function $T$ is well-defined as hypotheses $(\mathbb{H}1)$--$(\mathbb{H}4)$ together with 
Theorem \ref{extrMDE} guarantee the existence of extremal solutions of (\ref{FMDE2}) in $[\vec\alpha,\vec\beta]$.
Moreover, $T$ clearly satisfies \eqref{Tbound} with $h_i(t)=\int_{t_0}^tM_i(s)\, \dd g_i(s),$ 
$i\in\{1,\dots,n\}$. 
In order to apply Proposition \ref{fixedpoint}, we need to show that $T$ is nondecreasing. Consider 
$\vec\gamma,\vec\eta\in[\vec\alpha,\vec\beta]$ such that $\vec\gamma\leq\vec\eta.$ By 
hypothesis $(\mathbb{H}5)$ we have $\vec f(s,T\vec\eta(s),\vec\gamma)\leq \vec f(s,T\vec\eta(s),\vec\eta)$ for 
$s\in I$. Thus, for $i\in\{1,\dots,n\}$ and $[u,v]\subseteq I$ we get
$$\int_{u}^v (f_{\vec\gamma})_i(s,T\vec\eta(s))\,\dd g_i(s)
	\leq\int_{u}^v (f_{\vec\eta})_i(s,T\vec\eta(s))\,\dd g_i(s)=(T\vec\eta)_i(v)-(T\vec\eta)_i(u),$$
that is, $T\vec\eta$ is an upper solution of
\begin{equation}\label{FMDE3}
\vec z(t)=\vec{y_0}+\int_{t_0}^t\vec f_{\vec\gamma}(s,\vec z(s))\, \dd \vec g(s).
\end{equation}
Theorem \ref{extrMDE} guarantees that \eqref{FMDE3} has the greatest solution between $\vec\alpha$ and 
$T\vec\eta.$ Since $T\vec\gamma$ is the greatest solution of \eqref{FMDE3} in $[\vec\alpha,\vec\beta]$ it follows that $T\vec\gamma\leq T\vec\eta.$ Hence, $T$ is nondecreasing and Proposition \ref{fixedpoint} yields that $T$ has the greatest fixed point $\vec \gamma^*$ with 
$$\vec{\gamma}^*=\max\{\vec\gamma\in[\vec\alpha,\vec\beta]: \vec\gamma\leq T\vec\gamma\}.$$
Naturally, $\vec\gamma^*$ is a solution of \eqref{FMDE} in $[\vec\alpha,\vec\beta]$. Moreover, it is not hard to see that if $\vec\gamma\in[\vec\alpha,\vec\beta]$ is any other solution of \eqref{FMDE}, then 
$\vec\gamma\leq T\vec\gamma$. Therefore, by the definition of $\vec{\gamma}^*$, we conclude that $\vec{\gamma}^*$ is the greatest solution of \eqref{FMDE}.

To prove the existence of the least solution for equation \eqref{FMDE} we proceed in a similar way but redefining the function $T$ so that $T\vec\gamma$ corresponds to the least solution of equation \eqref{FMDE2}.
\end{proof}

Using the theorem above we can establish the existence of extremal solutions for 
measure functional differential equations:
\begin{equation}\label{mFDE}
\begin{aligned}
\vec y(t) &= \vec y(t_0) + \int_{t_0}^{t} \vec F(s,\vec y_s)\,{\rm d}g(s), \quad t \in I,\\
\vec y_{t_0} &= \vec \phi,\\
\end{aligned}
\end{equation}
where $I=[t_0,t_0+L]$, $r>0$, $\vec\phi\in G([-r,0],\R^n)$, $\vec F:I\times G([-r,0],\R^n)\to\mathbb R^n$ and 
$g:I\to\mathbb R$ is nondecreasing and left-continuous. 

\begin{theorem}\label{extFDE}
Let $J=[t_0-r,t_0]\cup I$.  
Let $\vec{\alpha},\,\vec\beta\in G(J,\R^n)$ be such that 
$\vec\alpha_{t_0}\leq\vec\phi\leq\vec\beta_{t_0}$ and 
$$\vec{\alpha}(v)-\vec{\alpha}(u) \le \int_{u}^{v}\vec{F}(s,\vec{\alpha}_s) \, \dd g(s),
\quad[u,v]\subseteq I,$$
$$\int_{u}^{v}\vec{F}(s,\vec{\alpha}_s) \, \dd g(s)\leq \vec{\beta}(v)-\vec{\beta}(u),
\quad[u,v]\subseteq I.$$
Assume that $\vec\alpha\leq\vec\beta$ and consider the functional interval 
$$[\vec{\alpha},\vec{\beta}]_J=\{ \vec{\eta} \in G(J,\mathbb R^n) \, : \, \vec{\alpha} \le \vec{\eta} \le \vec{\beta}\}.$$
Further, assume that the following conditions hold:
\begin{enumerate}
\item[\textup{(a)}] 
 The integral $\int_{t_0}^{t_0+L} F_i(s,\vec y_s)\, \dd g(s)$ exists 
for every $\vec y\in[\vec\alpha,\vec\beta]_J$ and $i\in\{1,\dots, n\}$. 
\item[\textup{(b)}] For each $i\in\{1,\dots, n\}$, there exists $M_i:I\to \R,$ which is Kurzweil-Stieltjes integrable with respect to $g$, such that
\[
\left|\int_u^v F_i(s,\vec y_s)\,{\rm d}g(s)\right|\leq \int_u^v M_i(s)\,{\rm d}g(s)
\]
for every $\vec\gamma\in[\vec{\alpha},\vec{\beta}]_J$ and $[u,v]\subseteq I$.
\item[\textup{(c)}] For each $t\in I$, the mapping $\varphi\in P\mapsto\vec F(t,\varphi)$ is nondecreasing, 
where $P\subset G([-r,0],\R^n)$ is the set $P=\{\vec y_s\,:\, \vec y\in [\vec\alpha,\vec\beta]_J,\,s\in I\}.$
\end{enumerate}
Then equation \eqref{mFDE} has extremal solutions in $[\vec\alpha,\vec\beta]_J.$
\end{theorem}

Note that assumptions $(\mathbb{H}3)$ and $(\mathbb{H}4)$ do not play a role in Theorem \ref{extFDE} 
as the function $\vec{f}(t,\vec{y}(t),\vec y)=\vec F(t,\vec y_t)$ does not depend on $\vec y(t)$.

\medskip

By setting $g(t)=t$, equation \eqref{mFDE} corresponds to the integral form of the 
retarded functional differential equation
\begin{equation}\label{FDE}
y'(t)=F(t,y_t),\quad t\in I,\qquad y_{t_0}=\phi.
\end{equation}
Regarding scalar equations \eqref{FDE}, the existence of solutions between well-ordered 
lower and upper solutions has been investigated in \cite{SZ}. Therein, a monotone interactive 
method is applied in order to obtain the extremal solutions. Although \cite{SZ} deals with 
lower/upper solutions which might be discontinuous, the function in the right-hand side, $F$, 
is assumed to satisfy the usual Car\'atheodory conditions. On one hand, 
in our Theorem \ref{extFDE} no continuity is required; however, the monotonicity condition 
(c) is admittedly stronger than the assumption (P5) stated at \cite[Theorem 4]{SZ}.

\section{Applications to Stieltjes differential equations}
Stieltjes differential equations are differential systems in which the usual notion of derivative is replaced by a 
differentiation process with respect to a given monotone function. The basic theory for such equations has been established in \cite{fp, pr}. In this work, we will consider vectorial Stieltjes differential equations of the form
\begin{equation}\label{Stielt}
\vec y\,'_{\vec g}(t)=\vec f(t,\vec y(t))\quad\mbox{for }\vec g\mbox{--a.a.} \, \, t\in I,\quad \vec y(t_0)=\vec y_0,
\end{equation}
where $I=[t_0,t_0+L]$, $\vec{y_0}\in \R^n$, $\vec{f}:I\times \R^n\to\mathbb R^n$ and $\vec g:I\to\mathbb R^n$ 
with $\vec g=(g_1,\dots,g_n)$ such that, for each $i\in\{1,\dots, n\}$, $g_i:I\to\R$ is nondecreasing and left-continuous. The problem described by  \eqref{Stielt} should be understood as the following system of Stieltjes differential equations:
\begin{equation}\label{Stieltsys}
(y_i)'_{g_i}(t)=f_i(t,\vec y(t))\quad\mbox{for } g_i\mbox{--a.a. } t\in I,
\quad y_i(t_0)=y_{0,i},\quad i\in\{1,\dots, n\}.
\end{equation}
For a thorough study of the Stieltjes derivative which appears in \eqref{Stieltsys} we refer to \cite{pr, fp}.

We remark that the equations studied in \cite{fp} are contained in \eqref{Stielt},   
corresponding to the particular choice $g_i=g:I\to\mathbb R$ for all $i\in\{1,\dots, n\}$. 
The Stieltjes equations in \cite{fp} were investigated in the space $\mathcal{AC}_{g}(I)$ of functions absolutely continuous with respect to $g$  nondecreasing and left-continuous. Recall that a function $y\in\mathcal{AC}_g(I)$ if for every $\varepsilon>0$ there exists $\delta>0$ such that
\[
\sum_{j=1}^m|y(b_j)-y(a_j)|<\varepsilon
\]
for any family $\{(a_j,b_j)\}$ of disjoint subintervals of $I$ satisfying 
$\sum_{j=1}^m(g(b_j)-g(a_j))<\delta$.
Extending the notion of solution found in \cite{fp}, we will look for solutions of the vectorial problem \eqref{Stielt} in the space
$$\mathcal{AC}_{\vec g}(I)=\mathcal{AC}_{g_1}(I)\times\dots\times\mathcal{AC}_{g_n}(I),$$ 
where $\vec g=(g_1,\dots,g_n):I\to\mathbb R^n$ is a nondecreasing left-continuous function. 

\begin{definition}\label{sol}
A solution of equation \eqref{Stielt} is a function $\vec y\in\mathcal{AC}_{\vec g}(I)$ such that \eqref{Stieltsys} holds.
\end{definition}

As a consequence of the Fundamental Theorem of Calculus for the Lebesgue--Stieltjes integral,  \cite{pr}, we have the following lemma.

\begin{lemma}
\label{lemanuevo}
Let $I=[t_0,t_0+L]$, $\vec{y_0}\in \R^n$, $\vec{f}:I\times \R^n\to\mathbb R^n$ and $\vec g:I\to\mathbb R^n$ 
with $\vec g=(g_1,\dots,g_n)$ such that, for each $i\in\{1,\dots, n\}$, $g_i:I\to\R$ is nondecreasing and left-continuous.

If $\vec y \in \mathcal{AC}_{\vec g}(I)$ is a solution of \eqref{Stielt}, then 
\begin{equation}
\label{intLS}
y_i (t)=y_{0,i}+\int_{[t_0,t)} f_i(s,\vec y(s)) \,\dd \mu_{g_i} \ \ \ \mbox{for all $t \in I$,} \quad i\in\{1,\dots,n\},
\end{equation}
where the integral stands for the Lebesgue-Stieltjes integral with respect to the Lebesgue-Stieltjes measure $\mu_{g_i}$ induced by $g_i$.

Conversely, if $\vec y=(y_1,\dots,y_n):I\to\mathbb R^n$ satisfies \eqref{intLS}, then $\vec y \in  \mathcal{AC}_{\vec g}(I)$ and it solves the vectorial Stieltjes differential equation \eqref{Stielt}.
\end{lemma}

Using the lemma above and recalling the relation between Lebesgue-Stieltjes and Kurzweil-Stieltjes integrals, \cite{MST}, one can show that a solution of \eqref{Stielt} is also a solution of the vectorial measure differential equation \eqref{MDE}.

In \cite{asat}, it is shown that, under very general assumptions, the integral equation \eqref{MDE1} is equivalent to
\begin{equation}\label{DDME}
y_g'(t)= f(t, y(t))\quad m_g\mbox{--a.e.}, \, \,\quad y(t_0)=y_0,
\end{equation}
where $m_g$ stands for the Thomson's variational measure (see $\mathcal S_{0^-}\mu_g$ in \cite{Th}) induced by a function $g:I\to\R$. In the case when $g$ is nondecreasing, as proved in \cite{ene}, the variational measure $m_g$ corresponds to the Lebesgue-Stieltjes outer measure $\mu_g^*$. Therefore, if $E\subset I$ and $m_g(E)=0,$ then $\mu_g^*(E)=0$ and, consequently, $E$ is Lebesgue-Stieltjes measurable with $\mu_g(E)=0.$ Accordingly, a solution of \eqref{DDME} also satisfies equation
\[
y_g'(t)= f(t, y(t))\quad\mbox{for } g\mbox{--a.a.} \, \, t\in I,\quad y(t_0)=y_0,
\]
 where $y \in \mathcal{AC}_{g}(I)$ if and only if $f(\cdot,y(\cdot))$ is integrable on $I$ with respect to $g$ in the Lebesgue-Stieltjes sense.
Therefore, along similar lines of the results in \cite{asat}, we can draw a correspondence between the solutions of 
$$y_i(t) = y_{0,i} + \int_{t_0}^{t} f_i(s,\vec{y}(s))\,{\rm d} g_i(s), \ \ \ t \in I, \ \ \ i\in\{1,\dots,n\},$$
and the solutions of
$$(y_i)'_{g_i}(t)=f_i(t,\vec y(t))\quad\mbox{for } g_i\mbox{--a.a. } t\in I, \ \ \ y_i(t_0)=y_{0,i}, \quadÊ i \in \{1,\dots,n\}.$$

Having all this in mind, based on results of previous sections, we can establish the existence of extremal solutions for Stieltjes differential equations \eqref{Stielt}. Note that extremal solutions to \eqref{Stielt} are defined in the obvious way in regard to Definition \ref{sol}. We now introduce the concepts of lower and upper solutions for this problem.

\begin{definition}
A lower solution of \eqref{Stielt} is a function $\vec \alpha\in\mathcal{AC}_{\vec g}(I)$ such that 
$\vec \alpha(t_0)\leq \vec y_0$ and 
\begin{equation}\label{lS}
(\alpha_i)'_{g_i}(t)\leq f_i(t,\vec \alpha(t))\quad\mbox{for } g_i\mbox{--a.a. } t\in I, \ \ \ i\in\{1,\dots,n\}.
\end{equation}
Analogously, $\vec \beta\in\mathcal{AC}_{\vec g}(I)$ is an upper solution of \eqref{Stielt} if 
$\vec y_0\leq\vec \beta(t_0)$ and 
\[
(\beta_i)'_{g_i}(t)\geq f_i(t,\vec \beta(t))\quad\mbox{for } g_i\mbox{--a.a. } t\in I, \ \ \ i\in\{1,\dots,n\}.
\]
\end{definition}

\begin{remark}
Every lower solution of \eqref{Stielt} is also a lower solution of \eqref{MDE1}. Indeed, given a lower solution $\vec\alpha$ of \eqref{Stielt}, since $\alpha_i\in\mathcal{AC}_{g_i}(I)$ for each  $i\in\{1,\dots,n\}$, by \cite[Theorem 5.1]{fp}, for every $[u,v]\subseteq I$ we have
$$\alpha_i(v)=\alpha_i(u)+\int_{[u,v)}(\alpha_i)'_{g_i}(s)\, \dd \mu_{g_i},$$
where the integral stands for the Lebesgue-Stieltjes integral with respect to the Lebesgue-Stieltjes measure $\mu_{g_i}$ induced by $g_i$.
Therefore, \eqref{lS} implies
$$\alpha_i(v)-\alpha_i(u)
\leq \int_{[u,v)}f_i(s,\vec\alpha(s))\, \dd \mu_{g_i}\quad [u,v]\subseteq I,
\quad i\in\{1,\dots,n\}.$$
Recall that Lebesgue-Stieltjes integrability implies Kurzweil-Stieltjes integrability, \cite{MST}. This, together with the fact that 
$g_i$ is left-continuous, ensures that the Lebesgue-Stieltjes integral on right-hand side coincides with the Kurzweil-Stieltjes integral $\int_{u}^vf_i(s,\vec\alpha(s))\, \dd g_i(s)$, see \cite{MST}. Since functions in the space $\mathcal{AC}_{g_i}(I)$ have bounded variation (\cite[Proposition 5.2]{fp}), we conclude that $\vec\alpha$ is a lower solution of the integral equation \eqref{MDE1}.

Similar arguments show that every upper solution of \eqref{Stielt} is also an upper solution of \eqref{MDE1}.
\end{remark}

In \cite{PM}, extremal solutions for \eqref{Stielt} have been studied in the scalar case. In order to apply the results of previous sections to investigate the solutions of the vectorial problem \eqref{Stielt} we will need the following lemma which corresponds to a particular case of \cite[Lemma 2.22]{asat}.

\begin{lemma}\label{null-int}
Let $g:[a,b]\to \mathbb{R}$ be nondecreasing and left-continuous. If $f:[a,b]\to \mathbb{R}$ is null $g$--a.e.,  then $\int_a^t f(s)\,\dd g(s)=0$ for every $t\in[a,b]$.
\end{lemma}

The following result, obtained from Theorem \ref{extrMDE}, ensures the existence of solution for the vectorial problem \eqref{Stielt} in the presence of lower and upper solutions.  

\begin{theorem}
\label{teorema}
Suppose that \eqref{Stielt} has a lower solution $\vec\alpha$ and an upper solution $\vec\beta$ such that $\vec{\alpha}\leq\vec{\beta}$. 
Assume that $\vec{f}$ is quasimonotone nondecreasing in $E=\{(t,\vec{x})\,:\,\vec{\alpha}(t)\leq\vec{x}\leq\vec{\beta}(t)\}$ and that the following conditions hold:
\begin{enumerate}
\item[{\rm(A)}] For each $i\in\{1,\dots,n\}$ we have
\begin{enumerate}
\item[\textup{(i)}] for every $\vec \eta \in [\vec \alpha, \vec \beta]$ the function $f_i(\cdot,\vec \eta(\cdot))$ is $g_i$-measurable;
\item[\textup{(ii)}] for each $\vec\eta\in[\vec\alpha,\vec\beta]$ and for $g_i$-a.a. $t\in I$, the mapping
$$u\in[\alpha_i(t),\beta_i(t)]\mapsto  f_i(t,\vec\eta(t)+(u-\eta_i(t))\vec{e_i}) $$
is continuous;
\item[\textup{(iii)}] for every $r>0,$ there exists a function $h_{i,r}:I\to\R_+$, which is Lebesgue-Stieltjes integrable with respect to $g_i$, such that
$$|f_i(t,\vec x)|\leq h_{i,r}(t),\quad\mbox{for } g_i\mbox{--a.a. } t\in I,
	\quad\mbox{for every }\vec x\in\mathbb R^n,\,\|\vec x\|\leq r.$$
\end{enumerate}
\item[{\rm(B)}] For each $\vec{\eta}\in[\vec{\alpha}, \vec{\beta}]$, $i\in\{1,\dots,n\}$, and $t\in I$, the mapping
\[
u\in[\alpha_i(t),\beta_i(t)]\mapsto u+f_i(t,\vec{\eta}(t)+(u-\eta_i(t))\vec{e_i})\Delta^+g_i(t)
\]
is nondecreasing.
\end{enumerate}
Then equation \textup{(\ref{Stielt})} has extremal solutions in 
$[\vec{\alpha},\vec{\beta}]\cap \mathcal{AC}_{\vec g}(I) $. Moreover, for $t\in I$, the greatest solution 
$\vec{y}^*=(y^*_1,\dots,y^*_n)$ is given by
\begin{equation*}
\label{xmaxsys'}
y^*_i(t)=\sup \{\ell_i(t) \, : \, \mbox{$(\ell_1,\dots,\ell_n)$ lower solution of \textup{(\ref{Stielt})} in 
$[\vec{\alpha},\vec{\beta}]\cap \mathcal{AC}_{\vec g}(I)$} \},
\end{equation*}
and the least solution $\vec{y}_*=(y_{*,1},\dots,y_{*,n})$ is given by
\begin{equation*}
\label{xminsys'}
y_{*,i}(t)=\inf \{u_i(t) \, : \, \mbox{$(u_1,\dots,u_n)$ upper solution of \textup{(\ref{Stielt})} in 
$[\vec{\alpha},\vec{\beta}]\cap \mathcal{AC}_{\vec g}(I)$} \}.
\end{equation*} 
\end{theorem}
\begin{proof}
For each $i\in\{1,\dots,n\}$, let $N_i\subset I$ be a $g_i$-null set such that both conditions (ii) and (iii) hold for $t\in I\setminus N_i$. Put $U_i:I\times\mathbb R^n\rightarrow\mathbb R$
\[
U_i(t,\vec x)=\left\{
\begin{array}{cl}
f_i(t, \vec x)&\mbox{if $t\in I\setminus N_i$, $\vec x\in\R^n$}\\
0&\mbox{otherwise.}
\end{array}
\right.
\]
Let $\vec U=(U_1,\dots,U_n):I\times\mathbb R^n\rightarrow\mathbb R^n$  and consider the modified problem
\begin{equation}\label{U-eq}
\vec y(t)=\vec y_0+\int_{t_0}^{t}{\vec U(s,\vec y(s)) \, \dd \vec g(s)} \quad t\in I.
\end{equation}
It follows from Lemma \ref{null-int} that a solution of \eqref{U-eq} is also a solution of \eqref{MDE}. Moreover, (A) guarantees that the integrals in (\ref{U-eq}) also exist as Lebesgue-Stieltjes integrals when $\vec y \in [\vec \alpha, \vec \beta]$. Hence, in view of Lemma \ref{lemanuevo}, the existence of extremal solutions for \eqref{U-eq} in $[\vec \alpha, \vec \beta]$ yields the existence of greatest and least solutions for \eqref{Stielt} in $[\vec \alpha, \vec \beta]$. Altogether, it is enough to show that the function 
$\vec U$ fulfills conditions $(\mathcal H1)-(\mathcal H4)$ of Theorem \ref{extrMDE}.  

Clearly, $\vec U$ satisfies $(\mathcal H3)$ and $(\mathcal H4)$. Note that (A) implies that for each $i\in\{1,\dots,n\}$ and $\vec\eta\in[\vec \alpha,\vec \beta]$, the Lebesgue-Stieltjes integral 
$\int_{[t_0,t_0+L)}U_i(s,\vec\eta(s))\, \dd \mu_{g_i}$ exists. Consequently, $(\mathcal H1)$ holds due to the relation between Lebesgue-Stieltjes and Kurzweil-Stieltjes integrals. 

To prove $(\mathcal H2)$, take  
$r=\max\{\|\vec\alpha\|_\infty\,,\,\|\vec\beta\|_\infty\}$ and let $h_{i,r}:I\to\R_+$, $i\in\{1,\dots,n\}$, be the corresponding function in (A)(iii). Since
$$|U_i(t,\vec x)|\leq h_{i,r}(t),\quad t\in I,
		\quad\vec x\in\mathbb R^n,\,\|\vec x\|\leq r,$$
and both functions are integrable with respect to $g_i$ in the sense of Kurzweil-Stieltjes, we get
\[
\left|\int_u^v U_i(t,\vec{x})\,{\rm d}g_i(t)\right|\leq \int_u^v h_{i,r}(t)\,{\rm d}g_i(t)
\]
for every $[u,v]\subseteq I$ and $\vec x\in\R^n$ with $\|\vec x\|\leq r.$ Proceeding as in the proof of Lemma 3.1 in \cite{as}, we can show that the inequality above still holds if we consider regulated functions 
$\vec x:I\to\R^n$ with $\|\vec{x}\|_\infty\leq r$. Thus, $(\mathcal H2)$ follows and this concludes the proof.
\end{proof}

Consider now the functional problem,
\begin{equation}\label{FStielt}
\vec y\,'_{\vec g}(t)=\vec f(t,\vec y(t),\vec y)\quad\mbox{for } \vec g\mbox{--a.a. }t\in I,\quad \vec y(t_0)=\vec y_0,
\end{equation}
where $\vec{y_0}\in \R^n$, $\vec{f}:I\times \R^n\times G(I,\R^n)\to\mathbb R^n$ and $\vec g:I\to\mathbb R^n$ is nondecreasing and left-continuous. Naturally, as before, \eqref{FStielt} denotes a system of Stieltjes differential equations subject to functional arguments; namely:
\begin{equation*}
(y_i)'_{g_i}(t)=f_i(t,\vec y(t),\vec y)\quad\mbox{for } g_i\mbox{--a.a. }t\in I,
\quad y_i(t_0)=\vec y_{0,i},\quad i\in\{1,\dots,n\}.
\end{equation*}
A solution of this functional equation is defined analogously to those of problem \eqref{Stielt}, and so are the  upper and lower solutions. 

Extremal solutions for \eqref{FStielt} were recently investigated in \cite{PM} in the scalar case. Next, applying Theorem \ref{funcexist}, we present a result for \eqref{FStielt} in its general formulation. Such a result relies on a correspondence between \eqref{FStielt} and \eqref{FMDE} which can be obtained from an extension of the argument used in \cite{asat}. 

\begin{theorem}\label{Teorema2}
Suppose that \eqref{FStielt} has a lower solution $\vec\alpha$ and an upper solution $\vec\beta$ such that $\vec{\alpha}\leq\vec{\beta}$ and let 
$E=\{(t,\vec{x})\in I\times\R^n\,:\,\vec{\alpha}(t)\leq\vec{x}\leq\vec{\beta}(t)\}$. 
For each $\vec\gamma\in[\vec\alpha,\vec\beta]$, denote by 
$\vec f_{\vec\gamma}:I\times\mathbb R^n\rightarrow\mathbb R^n$ the function defined as  
$\vec f_{\vec\gamma}(t,\vec x)=\vec f(t,\vec x,\vec\gamma)$. 
Assume that for each $\vec\gamma\in[\vec\alpha,\vec\beta]$ the function $\vec f_{\vec\gamma}$ satisfies the conditions in Theorem \ref{teorema}. If for $\vec g$--a.a. $t\in I$ and all $\vec x\in\mathbb R^n$, the mapping $\vec f(t,\vec x,\cdot)$ is nondecreasing on $[\vec{\alpha},\vec{\beta}]$,
then problem \eqref{FStielt} has extremal solutions in $[\vec\alpha,\vec\beta]\cap \mathcal{AC}_{\vec g}(I).$
\end{theorem}

\section{A simple model for a bacteria population with variable carrying capacity}

Consider an open tank which contains an initial amount of water reaching a level of $L$ meters high and assume that the changes on the level of water are exclusively caused by evaporation as a result of the effect of the sun. According to this, during the day the level of water will change, whereas it will remain constant during night hours. Now consider a bacteria population whose resources depend directly on the volume of water. This means that the carrying capacity (that is, the number of bacteria that can be supported indefinitely in the tank) will be dependent on the level of water: the higher the level of water is, the bigger the carrying capacity will be. Finally, we will also assume that every morning, the tank is refilled until a certain level depending on the population of bacterias at that time. 

We want to design a mathematical model for $w(t),$ the water height at time $t>0$, and $p(t)$, the bacteria population at time $t>0$, under the previous assumptions. For the latter, we will consider a logistic model where the carrying capacity will be given by a nondecreasing function $N:\mathbb R\to\mathbb R$ depending on $w(t)$. Hence, the population $p$ at time $t$ is represented by the equation
$$p'(t)=r p(t) (N(w(t))-p(t)),$$
with $r>0$ being the reproduction rate of the population.

In order to find an expression that suits $w(t)$, we want to differentiate with respect to a function $g$ which is constant during night hours and such that it assigns greater measure to middays, when the effect of the sun is stronger. We also want $g$ to present jump discontinuities at the beginning of each day in order to introduce instantaneous changes in $w(t)$ due to refillings. Thus, if we identify day hours with the intervals $[2k,2k+1]$ for $k=0,1,2,\dots$ and nights with $[2k+1,2k+2]$, $k=0,1,2,\dots$, a possible choice is
$$g(t)=\max\{ k=0,1,2,...: 2k\leq t\}+\int_0^t \max\{\sin(\pi s),0\}\,\dd s,\quad t\in\mathbb R,$$
since it is constant on $[2k+1,2k+2]$, $k=0,1,2,\dots,$ presents jump discontinuities at times $t\in 2\,\mathbb N$, and has maximum slopes at $t=1/2+2k,$ $k=1,2,\dots$ (which represent middays).

\begin{figure}
	\centering
	\subfigure[Graph of $g(t)$.]{\includegraphics[width=7.5cm]{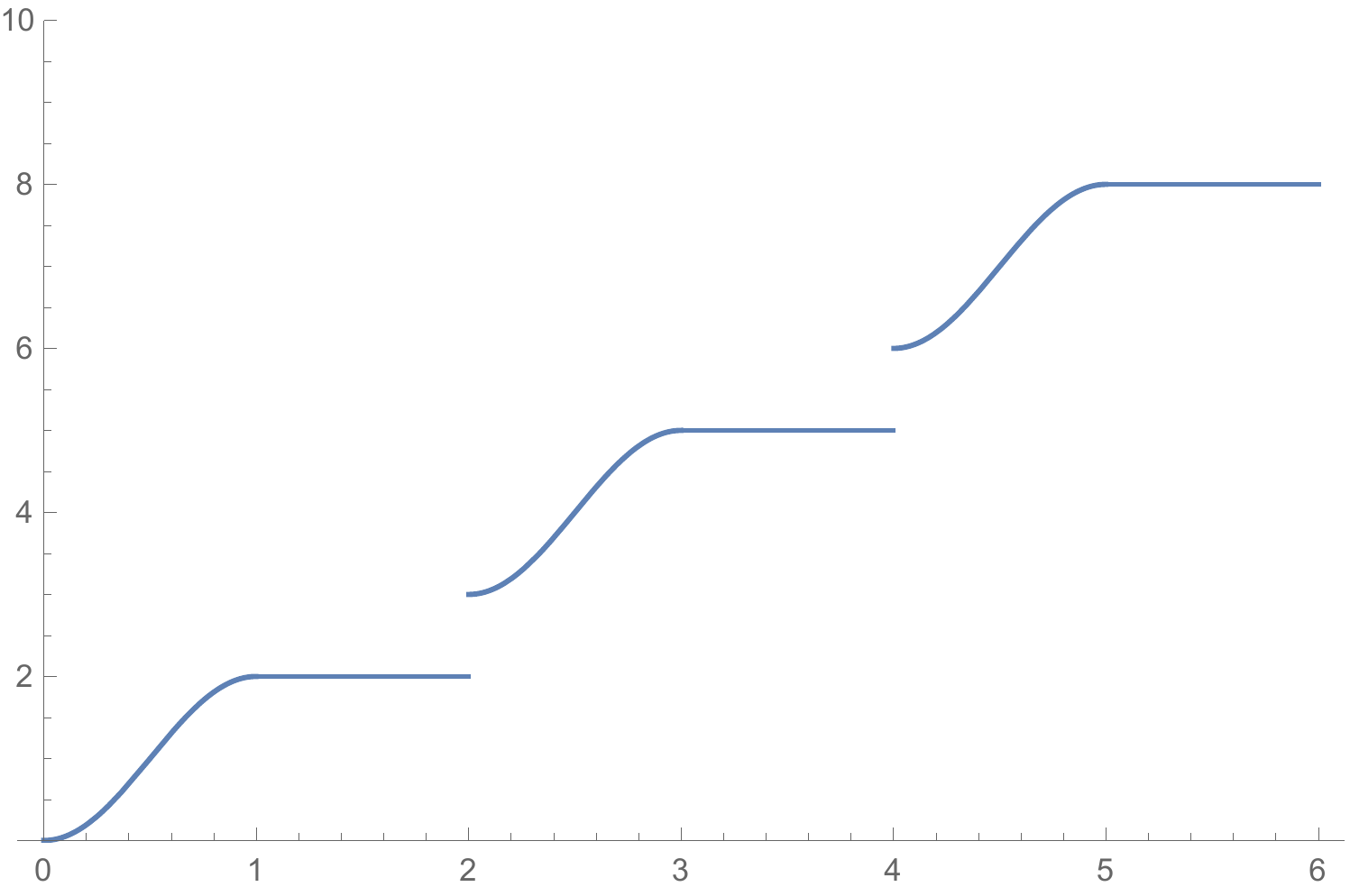}}\quad
	\subfigure[Graph of $W(t)$ for $L=10,$ $c=\pi$.]{\includegraphics[width=7.5cm]{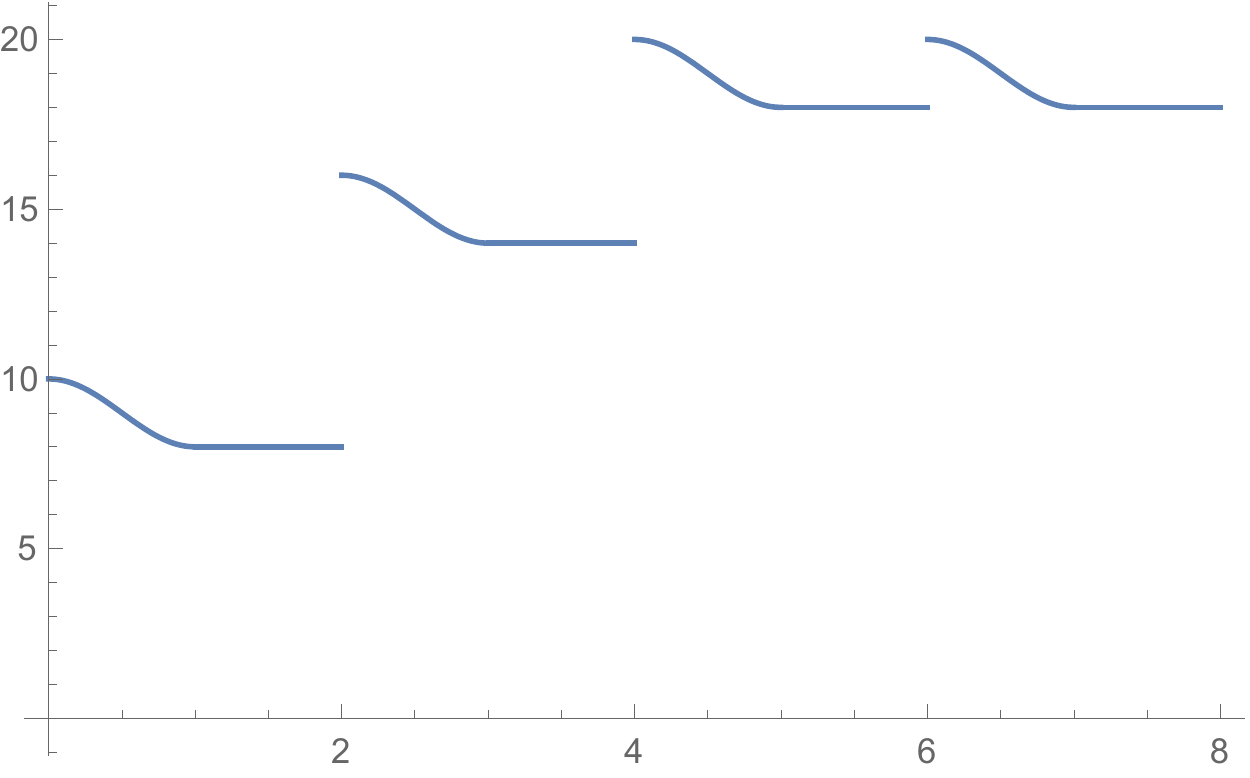}}
\end{figure}

Let $c>0$ be the evaporation rate of the water and let $a\in\R$ be a proportionality parameter readjusting the mistakes arising from counting the number of bacteria. Fixed an arbitrary time $T>0$, consider the described problem in the span interval $I=[0,T]$. A simple first model for $w(t)$ is given by 
\[
w'_g(t)=F(t,p(t),w(t)),
\]
where $F:I\times \mathbb R\times \mathbb R\to\mathbb R$ is defined by
\[
F(t,p,w)=\left\{
\begin{array}{ll}
\min\{\lfloor a\,p\rfloor w, 2L-w\},&\mbox{if $t\in 2\mathbb N$,}\\[2mm]
-c,&\mbox{otherwise,}\\
\end{array}
\right.
\]
and  $\lfloor x\rfloor$ stands for the greatest integer part of a number $x$. 
Note that the function $F$ defined in this way at times $t\in 2\mathbb N$ implies $w(2k)\leq w(2k+)\leq 2L,$ $k\in\mathbb N$.
Indeed,
$$w'_g(2k)=w(2k+)-w(2k)=\min\{\lfloor \alpha\, p(2k)\rfloor w(2k), 2L-w(2k)\}$$
so $w(2k+)=\min\{(1+\lfloor \alpha\, p(2k)\rfloor) w(2k), 2L\}.$

Therefore, we consider the following system of differential equations
\begin{equation}\label{bact1}
\left\{
\begin{array}{rcll}
p'(t)&=&rp(t)(N(w(t))-p(t)),& \quad p(0)=p_0,\\[2mm]
w'_g(t)&=&F(t,p(t),w(t)),&\quad w(0)=L,
\end{array}
\right.
\end{equation}
which can be regarded as a vectorial Stieltjes differential equation \eqref{Stielt} with unknown term $\vec y=(p,w)$, initial condition $\vec y_0=(p_0,L)$, and the functions 
$\vec g:I\to \R^2$, $\vec f:I\times\R^2\to\R^2$ are given by
\begin{equation}\label{vec-f}
\vec g(t)=(t,g(t)),\quad \vec{f}(t,(p,w))=(rp(N(w)-p)\,,\,F(t,p,w)).
\end{equation}

We will use Theorem \ref{teorema} to show that \eqref{bact1} has at least one solution. Clearly, $\vec\alpha(t)=(0,0)$ is a lower solution of \eqref{bact1}. On the other hand, if we consider the map $W:I\rightarrow\mathbb R$ given by
\[
W(t)=
\left\{\begin{array}{ll}
L-c\int_0^t \max\{\sin(\pi s),0\}\,\dd s, &\mbox{$t\in [0,2],$}\\[2mm]
\min\{(1+\lfloor a\, p(2k)\rfloor)w(2k)\,,2L\}-c\int_{2k}^t \max\{\sin(\pi s),0\}\,\dd s, &\mbox{$t\in (2k,2k+2],\,k\in\mathbb N$},\\
\end{array}
\right.
\]
then $\vec{\beta}(t)=(p_0\, \exp\{\int_0^t rN(W(s))\,\dd s\}\,,\,W(t))$ is a solution of 
\begin{equation}\label{bact2}
\left\{
\begin{array}{rcll}
p'(t)&=&rp(t)N(w(t)),& \quad p(0)=p_0,\\[2mm]
w'_g(t)&=&F(t,p(t),w(t)),&\quad w(0)=L,
\end{array}
\right.
\end{equation}
and therefore, an upper solution of \eqref{bact1}.

Clearly, the map $\vec f$ defined in \eqref{vec-f} 
is quasimonotone nondecreasing in $I\times\mathbb R^2,$ and in particular in $E=\{(t,\vec{x})\,:\,\vec{\alpha}(t)\leq\vec{x}\leq\vec{\beta}(t)\}.$ Moreover, it is easy to check that $\vec{f}$ satisfies hypotheses (A)--(B) of Theorem \ref{teorema}, therefore the problem \eqref{bact1} has the extremal solutions between $\vec \alpha$ and $\vec \beta.$

So far, we have only used the fact that $N(w)$ is a nondecreasing function of $w$, so no matter if it is continuous or not, our theory applies. However, in some cases we may
find it reasonable to allow the carrying capacity $N(w)$ to be piecewise constant because very small changes in the water level could have no influence on the carrying capacity. A simple example appears when we consider $N$ to be the floor function, $N(t)=\lfloor t\rfloor.$ As we mentioned before, $W(t)$ is a solution of
$$w'_g(t)=F(t,p(t),w(t)),\quad w(0)=L.$$
Hence, we obtain the following ODE
\begin{equation}\label{floor}
p'(t)=rp(t)(\lfloor W(t)\rfloor-p(t)),\quad p(0)=p_0.
\end{equation}
One can easily check that $\lfloor W(t)\rfloor$ has at most a countable number of discontinuities, which we will denote by $\{t_i\}_{i\in\mathbb N}.$ Put $t_0=0.$ Note that equation \eqref{floor} is a linear equation in each interval $(t_i,t_{i+1}]$, $i=0,1,2,\dots,$ which can be solved exactly. 
 Setting $p_i=p(t_i)$, $i\in\N$, then the solution of \eqref{floor} is
$$p(t)=\frac{e^{\lfloor W(t)\rfloor r t}}{e^{\lfloor W(t_i)\rfloor r t_i}\left(\frac{1}{p_i}-\frac{1}{\lfloor W(t_i)\rfloor}\right)+\frac{e^{\lfloor W(t)\rfloor r t}}{\lfloor W(t)\rfloor}},\quad t\in(t_i,t_{i+1}].$$

\begin{figure}
	\centering
	\subfigure[Solution of \eqref{floor} for $L=10,$ $c=\pi$, $a=1/7$, $r=1$ and $p_0=5$.]{\includegraphics[width=7.5cm]{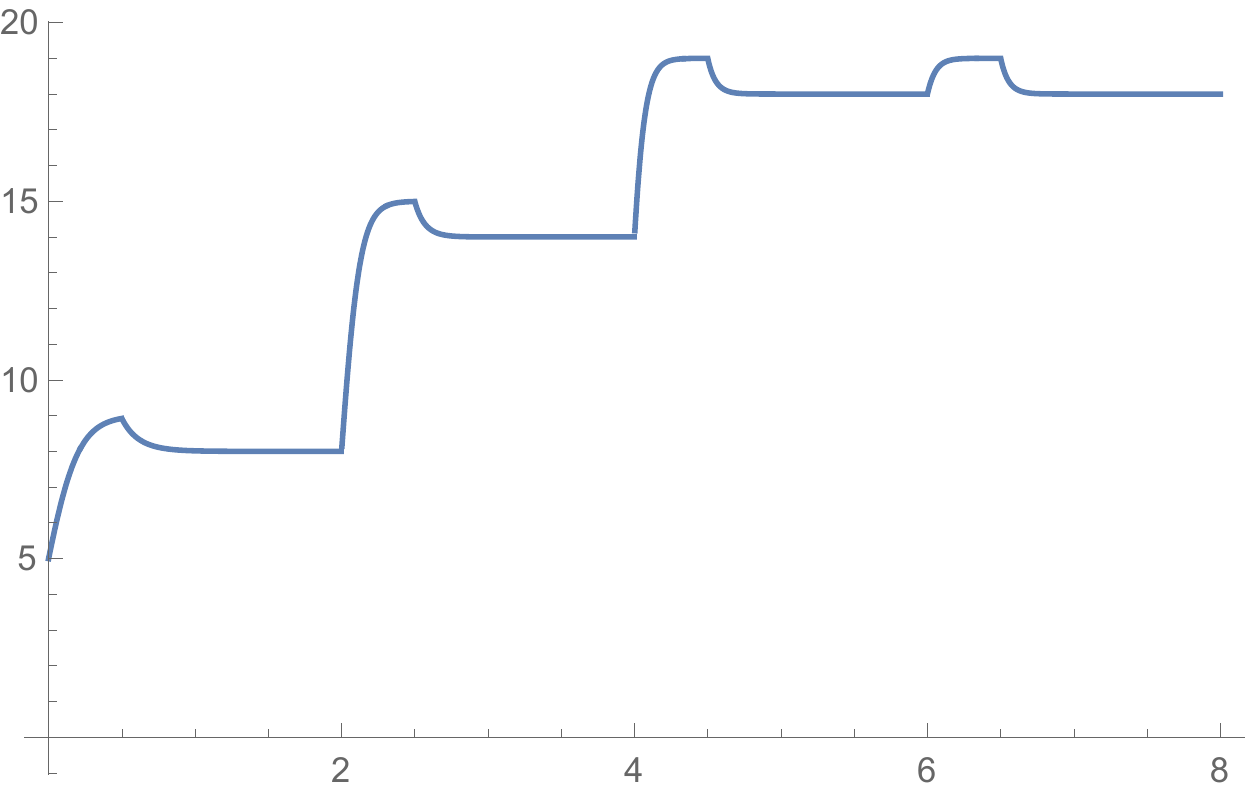}}\quad
	\subfigure[Solution of \eqref{floor} for $L=10,$ $c=\pi$, $a=1/7$, $r=1$ and $p_0=15$.]{\includegraphics[width=7.5cm]{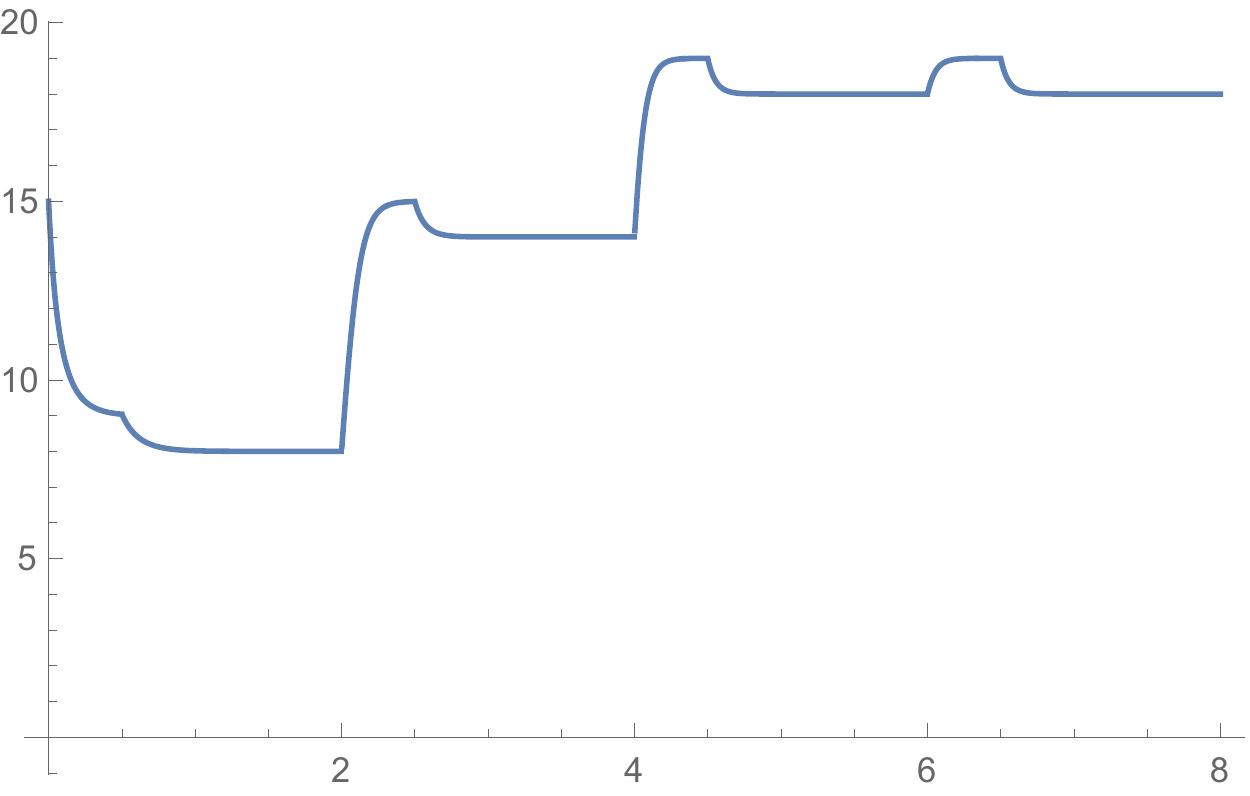}}
\end{figure}

Our theory also applies when we consider the following modified version of problem \eqref{bact1} with a functional argument:
\begin{equation}\label{bact3}
\left\{
\begin{array}{rcll}
p'(t)&=&rp(t)(N(w(t))-p(t)),&\quad p(0)=p_0,\\[2mm]
w'_g(t)&=&\tilde F(t,w(t),p),&\quad w(0)=L,
\end{array}
\right.
\end{equation}
where the map $\tilde F:I\times \mathbb R\times L^1(I)\to\mathbb R$ is given by 
\[
\tilde F(t,w,\varphi)=\left\{
\begin{array}{ll}
\min\{\lfloor a\, \int_{2(n-1)}^{2n} \varphi(s)\,\dd s\rfloor w\,, 2L-w\},&\mbox{if $t=2n$, $n \in \mathbb N$,}\\[2mm]
-c,&\mbox{otherwise.}
\end{array}
\right.
\]
In this case, $\vec{\alpha}=(0,0)$ is a lower solution of \eqref{bact3} and an upper solution can be obtained analogously  to the one from problem \eqref{bact1}. Hence, applying Theorem \ref{Teorema2} we can ensure the existence of solution for \eqref{bact3}.

\section*{Acknowledgement}
The project was financed by the SASPRO Programme. The research leading to these results has received funding from the People Programme (Marie Curie Actions) European Union's Seventh Framework Programme under REA grant agreement No. 609427. Research has been further co-funded by the Slovak Academy of Sciences.

\noindent The Institute of Mathematics of the Czech Academy of Sciences is supported by RVO:67985840.
\noindent
Rodrigo L\'opez Pouso was partially supported by
Ministerio de Econom\'{\i}a y Competitividad, Spain, and FEDER, Project
MTM2016-75140-P, and Xunta de Galicia REDES 2016 GI-1561 IEMath-Galicia and GRC2015/004

\noindent
Ignacio M\'arquez Alb\'es was partially
supported by Xunta de Galicia under grant ED481A-2017/095


\begin{thebibliography}{yyy}

\bibitem{Biles-Pouso} D.~C.~Biles, R.~L.~Pouso, \emph{First-order singular and discontinuous differential equations}, 
Bound.~Value Probl.~2009, Art.~ID 507671, 25 pp.

\bibitem{Biles-Schechter} D.~C.~Biles, E.~Schechter, \emph{Solvability of a finite or infinite system of discontinuous quasimonotone
differential equations}, Proc.~Amer.~Math.~Soc. {\bf 128} (2000), no.~11, 3349--3360.


\bibitem{caotpo} A. Cabada, V. Otero-Espinar, R. L. Pouso, \emph{Existence and approximation of solutions for first-order discontinuous difference equations with nonlinear global conditions in the presence of lower and upper solutions}, Comput. Math. Appl. {\bf 39} (2000), no. 1--2, 21--33.

\bibitem{Cid} J.~\'A.~Cid, \emph{On extending existence theory from scalar ordinary differential equations to infinite quasimonotone systems of functional equations}, Proc.~Amer.~Math.~Soc. {\bf 133} (2005), no.~9, 2661--2670.

\bibitem{ene}  V. Ene, \emph{Thomson's variational measure and some classical theorems}, Real Anal. Exchange   {\bf 25} (1999), 521--546.


\bibitem{fed} M. Federson, J. G. Mesquita, A.
Slav\'{\i}k, \emph{Measure functional differential equations and functional
dynamic equations on time scales}, J. Differential Equations {\bf 252} (2012), 3816--3847.

\bibitem{fed2} M. Federson, J. G. Mesquita, A.
Slav\'{\i}k, \emph{Basic results for functional differential and dynamic
equations involving impulses}, Math. Nachr. {\bf 286}(2-3) (2013), 181--204.

\bibitem{F} D. Fra\v nkov\'a, \emph{Regulated functions}, Math.~Bohem. {\bf 116}(1) (1991), 20--59.

\bibitem{fp} M. Frigon, R. L\'opez Pouso, \emph{Theory and applications of first-order systems of Stieltjes differential equations}, Adv. Nonlinear Analysis {\bf 6} (2017), 13--36.

\bibitem{HR} E.~R.~Hassan, W.~Rzymowski, 
\emph{Extremal solutions of a discontinuous scalar differential equation},
Nonlinear Anal. {\bf 37} (1999), no. 8, 997--1017.


\bibitem{lihela} {S. Heikkil\"a, V. Lakshmikantham,} {\sl Monotone
iterative techniques for discontinuous nonlinear differential equations}. 
Marcel Dekker, New York, 1994.

\bibitem{H}
    C. S. H\"onig,
    {\sl Volterra-Stieltjes integral equations.}
    North Holland and American Elsevier, Mathematics Studies 16, Amsterdam and New York, 1975.


\bibitem{lp} R. L\'opez Pouso, \emph{Peano's existence theorem revisited}, arXiv:1202.1152 [math.CA].

\bibitem{Pouso6} R.~L\'opez Pouso, \emph{Upper and lower solutions for first-order discontinuous ordinary differential equations}, J.~Math.~Anal.~Appl. {\bf 244} (2000), no.~2, 466--482.

\bibitem{Pouso4} R. L\'opez Pouso, 
\emph{Nonordered discontinuous upper and lower solutions for first-order ordinary differential equations}, Nonlinear Anal. {\bf 45} (2001), no.~4, 391--406.

\bibitem{PM} R. L\'opez Pouso, I. M\'arquez Alb\'es, \emph{General existence principles for Stieltjes differential equations with applications to mathematical biology}, to appear in \emph{J. Differential Equations}.

\bibitem{pr} R. L\'opez Pouso, A. Rodr\'{\i}guez, \emph{A new unification of continuous, discrete, and impulsive calculus through Stieltjes derivatives}, Real Anal. Exchange {\bf 40} (2014/15), no. 2, 1--35.

\bibitem{asat}  G. A. Monteiro, B. Satco, \emph{Distributional, differential and integral problems: Equivalence and existence results}, Electron. J. Qual. Theory Differ. Equ. {\bf 7} (2017), 1--26.

\bibitem{as}  G. A. Monteiro and A. Slav\'{\i}k, \emph{Extremal solutions of measure differential equations}, J. Math. Anal. Appl. {\bf 444} (2016), 568--597. 

\bibitem{MS} G.~A.~Monteiro, A.~Slav\'{\i}k, \emph{Linear measure functional differential equations with infinite delay},  Math. Nachr.  {\bf 287} (2014), 1363--1382. 

\bibitem{MST}  G. A. Monteiro, A. Slav\'{\i}k, M. Tvrd\' y, {\sl Kurzweil-Stieltjes integral: theory and applications}. Series in Real Analysis - vol 14, World Scientific, to appear in 2017

\bibitem{Peano} G.~Peano, \emph{Sull' integrabilit\`a delle equazioni differenziali di primo ordine}, Atti Acad.~Sci.~Torino 21 (1886), 677--685.

\bibitem{STV}
\v{S}.~Schwabik, M.~Tvrd\'y, O.~Vejvoda,
{\sl Differential and integral equations: boundary value problems and adjoints}. 
Academia and  D. Reidel, Praha and Dordrecht, 1979.



\bibitem{SZ} W. Sun, S. Zhang, \emph{Existence and approximation of solutions for
discontinuous functional differential
equations}, J. Math. Anal. Appl. {\bf 270} (2002) 307--318

\bibitem{Th} B. S. Thomson, {\sl Real functions}, Lect. Notes in Math., vol. 1170, Springer-
Verlag, 1985.

\bibitem{Tv5}
M. Tvrd\'y, \emph{Differential and integral equations in the space of regulated functions}, 
    Mem. Differential Equations Math. Phys. 25 (2002), 1--104.


 \end{thebibliography}
\end{document}